\documentclass{amsart}
\usepackage[usenames,dvipsnames]{xcolor}
\definecolor{darkgreen}{rgb}{0,0.5,0}
\usepackage{url}
\usepackage[
        colorlinks, citecolor=darkgreen,
        backref,
        pdfauthor={Jennifer S. Balakrishnan, Netan Dogra, J. Steffen M\"uller, Jan
        Tuitman, Jan Vonk},
        pdftitle={Explicit Chabauty--Kim for the Split Cartan Modular Curve of Level 13}
]{hyperref} 

\usepackage{tikz,enumerate,caption,stmaryrd,amsfonts,amssymb,systeme,comment}
\usetikzlibrary{matrix,positioning,arrows,shapes,chains,calc,automata}
\usetikzlibrary{decorations.pathreplacing}
\usepackage[T1]{fontenc}

\newlength{\proofmargin}
\renewcommand*{\backref}[1]{}
\renewcommand*{\backrefalt}[4]{%
  \ifcase #1 %
    \relax
  \or
    $\uparrow$#2.%
  \else
    $\uparrow$#2.%
  \fi%
}

  \renewenvironment{thebibliography}[1]{%
    \begin{oldthebibliography}{#1}%
      \setlength{\parskip}{0ex}%
      \setlength{\itemsep}{0.1ex}%
      \setlength{\labelwidth}{1.4cm} 
  }%
  {%
    \end{oldthebibliography}%
  }

\DeclareFontFamily{U}{wncy}{}
\DeclareFontShape{U}{wncy}{m}{n}{<->wncyr10}{}
\DeclareSymbolFont{mcy}{U}{wncy}{m}{n}
\DeclareMathSymbol{\Sha}{\mathord}{mcy}{"58}
\DeclareMathOperator{\NS}{\mathrm{NS}}

\DeclareMathOperator{\Q}{\mathbf{Q}}
\DeclareMathOperator{\F}{\mathbf{F}}
\DeclareMathOperator{\Z}{\mathbf{Z}}

\DeclareMathOperator{\D}{\mathbf{D}}
\DeclareMathOperator{\cA}{\mathcal{A}}
\DeclareMathOperator{\cX}{\mathcal{X}}
\DeclareMathOperator{\cY}{\mathcal{Y}}
\DeclareMathOperator{\cC}{\mathcal{C}}
\DeclareMathOperator{\PP}{\mathbf{P}}
\DeclareMathOperator{\HH}{\mathrm{H}}
\DeclareMathOperator{\rk}{\mathrm{rk}}
\DeclareMathOperator{\U}{\mathrm{U}}
\DeclareMathOperator{\Mf}{\mathrm{M}_f} 
\DeclareMathOperator{\cMf}{\mathcal{M}_f}

\DeclareMathOperator{\et}{\mbox{\scriptsize \'et}}
\DeclareMathOperator{\nr}{\mbox{\scriptsize nr}}

\DeclareMathOperator{\sC}{\mbox{\scriptsize s}}
\DeclareMathOperator{\dR}{\mbox{\scriptsize dR}}
\DeclareMathOperator{\an}{\mbox{\scriptsize an}}

\DeclareMathOperator{\fil}{\mbox{\scriptsize Fil}}
\DeclareMathOperator{\cris}{\mbox{\scriptsize cris}}
\DeclareMathOperator{\ns}{\mbox{\scriptsize ns}}

\DeclareMathOperator{\rig}{\mbox{\scriptsize{rig}}}

\DeclareMathOperator{\Hf}{\HH^1_{f}}
\DeclareMathOperator{\spec}{\mathrm{Spec}}

\DeclareMathOperator{\Ext}{\mathrm{Ext}}
\DeclareMathOperator{\Aut}{\mathrm{Aut}}

\DeclareMathOperator{\Sym}{\mathrm{Sym}}

\DeclareMathOperator{\Fr}{\mathrm{Fr}}
\DeclareMathOperator{\Hom}{\mathrm{Hom}}
\DeclareMathOperator{\End}{\mathrm{End}}

\DeclareMathOperator{\Gal}{\mathrm{Gal}}

\DeclareMathOperator{\Coker}{\mathrm{Coker}}
\DeclareMathOperator{\Fil}{\mathrm{Fil}}

\DeclareMathOperator{\Pic}{\mathrm{Pic}}
\DeclareMathOperator{\AJb}{\mathrm{AJ}_b}
\DeclareMathOperator{\JacX}{\mathrm{J}}
\DeclareMathOperator{\loc}{\mathrm{loc}}
\DeclareMathOperator{\Sel}{\mathrm{Sel}}
\DeclareMathOperator{\GL}{\mathrm{GL}}
\DeclareMathOperator{\spf}{\mathrm{Spf}}
\DeclareMathOperator{\cl}{\mathrm{cl}}
\DeclareMathOperator{\lra}{\longrightarrow}
\newcommand{\ra}{\rightarrow}

\newcommand{\Ker}{\mathrm{Ker}}

\newcommand{\nice}{nice } 

\newtheorem{theorem}{Theorem}[section]
\newtheorem{prop}[theorem]{Proposition}

\newtheorem*{question}{Question}
\newtheorem{lemma}[theorem]{Lemma}
\newtheorem{corollary}[theorem]{Corollary}

\newtheorem{Definition}[theorem]{Definition}

\theoremstyle{remark}
\newtheorem{Remark}[theorem]{Remark}

\makeatletter
\renewenvironment{proof}[1][\proofname]%
{%
\par\pushQED{\qed}\normalfont\topsep6\p@\@plus6\p@\relax%
\begin{list}{}{\rightmargin=8pt\leftmargin=\proofmargin}%
  \item[\hskip\labelsep\bfseries#1\@addpunct{.}]\ignorespaces
}{%
\popQED\end{list}\@endpefalse%
}%
\makeatother

\title{Explicit Chabauty--Kim for the Split Cartan Modular Curve of Level 13}
\author{Jennifer S. Balakrishnan}
\address{Jennifer S. Balakrishnan, Department of Mathematics and Statistics, Boston University, 111 Cummington Mall, Boston, MA 02215, USA}
\email{jbala@bu.edu}
\author{Netan Dogra}
\address{Netan Dogra, Department of Mathematics, Imperial College London, London SW7 2AZ, UK }
\email{n.dogra@imperial.ac.uk}
\author{J. Steffen M\"uller}
\address{J. Steffen M\"uller,
  Johann Bernoulli Institute, 
  University of Groningen
  Nijenborgh 9,
  9747 AG Groningen,
  The Netherlands
}
\email{steffen.muller@rug.nl}
\author{Jan Tuitman}
\address{Jan Tuitman, KU Leuven,
         Departement Wiskunde,
         Celestijnenlaan 200B,
         3001 Leuven,
         Belgium}
\email{jan.tuitman@kuleuven.be}
\author{Jan Vonk}
\address{Jan Vonk, Department of Mathematics and Statistics, McGill University, Montr\'eal H3A 0B9, Canada}
\email{jan.vonk@math.mcgill.com }

\date{}

\begin{document}

\begin{abstract} We extend the explicit quadratic Chabauty methods developed in
  previous work by the first two authors
  to the case of non-hyperelliptic curves. 
  This results in an algorithm to compute the rational points on a curve of genus $g \ge
  2$ over the rationals whose Jacobian has Mordell-Weil rank $g$ and Picard number
  greater than one, and which satisfies some additional conditions.
This algorithm is then applied to the modular curve $X_{\sC}(13)$, completing the
classification of non-CM elliptic curves over $\Q $ with split Cartan level structure due to
Bilu--Parent and Bilu--Parent--Rebolledo.\end{abstract}
\maketitle
\tableofcontents


\section{Introduction}

In this paper, we explicitly determine the rational points on $X_{\sC}(13)$, a genus 3 modular curve defined over $\Q$ with simple Jacobian having Mordell-Weil rank 3. This computation makes explicit various aspects of Minhyong Kim's nonabelian Chabauty programme and finishes the ``split Cartan'' case of Serre's uniformity
question on residual Galois representations of elliptic curves. Moreover, the broader techniques are potentially of interest for determining rational points on other curves. We begin with an overview of Serre's question, outline our strategy to compute $X_{\sC}(13)(\Q)$ in the context of Kim's nonabelian Chabauty, and end with some remarks on the scope of the method in the toolbox for explicitly determining rational points on curves.

\subsection{Modular curves associated to residual representations of elliptic curves}
\par If $E/\Q$ is an elliptic curve and $\ell$ is a prime number, then there is a natural
residual Galois representation
\[
  \rho_{E,\ell}:\Gal(\bar{\Q}/\Q)\to \Aut(E[\ell])\simeq \GL_2(\F_\ell).
\]
Serre~\cite{Ser72} showed that if $E$ does not have complex multiplication (CM), then 
$\rho_{E,\ell}$ is surjective for all primes $\ell \gg 0$.
\begin{question} [Serre]
Is there a constant $\ell_0$ such that $\rho_{E,\ell}$ is surjective for all elliptic curves
$E/\Q$ without CM and all primes $\ell>\ell_0$?
\end{question}
It is well-known that if such a constant $\ell_0$ exists, then it must be at least~37. 
To tackle this question, one uses that the maximal proper subgroups of $\GL_2(\F_\ell)$ are 
either Borel subgroups, normalizers of (split or non-split) Cartan subgroups, or exceptional subgroups.
The Borel and the exceptional cases were handled by Mazur~\cite{Maz78} and
Serre~\cite{Ser72}, respectively, and the case of normalizers of split Cartan subgroups
(for $\ell>13$) follows from
results of Bilu-Parent~\cite{BP11} and Bilu-Parent-Rebolledo~\cite{BPR13}, which we now recall.

\par
For a prime $\ell$, we write $X_{\sC}(\ell)$ for the modular curve $X(\ell)/C_{\sC}(\ell)^+$, 
where $C_{\sC}(\ell)^+$ is the normalizer of a split Cartan subgroup of $\GL_2(\F_\ell)$.
Since all such subgroups $C_{\sC}(\ell)^+$ are conjugate, $X_{\sC}(\ell)$ is well-defined up to
$\Q$-isomorphism.
Bilu--Parent~\cite{BP11} proved the existence of a constant $\ell_{\sC}$ such
that $X_{\sC}(\ell)(\Q)$ only consists of cusps and CM points for all primes $\ell>\ell_{\sC}$. 
This was later improved by Bilu-Parent-Rebolledo~\cite{BPR13} who showed that the
statement holds for all $\ell >7$, $\ell \ne 13$.
This proves that, for all primes $\ell>7$, $\ell \ne 13$, there exists no elliptic curve $E/\Q$ without CM
whose mod-$\ell$ Galois representation has image contained in
the normalizer of a split Cartan subgroup of $\GL_2(\F_\ell)$. However, they were unable to
prove this statement for $\ell=13$.

\par  Bilu--Parent and  Bilu--Parent--Rebolledo use a clever combination of several techniques
for finding all rational points on the curves $X_{\sC}(\ell)$,
but one of the crucial ingredients is an application of Mazur's method~\cite{Maz78} to show an
integrality result for non-cuspidal rational points on $X_{\sC}(\ell)$.
The latter relies on the statement 
\[ \mathrm{Jac}(X_{\sC}(\ell)) \sim \mathrm{Jac}(X_{0}^+(\ell^2)) \sim J_0(\ell) \times
\mathrm{Jac}(X_{\ns}(\ell)) \]
proved by Momose \cite{Mom86},
where $X_{\ns}(\ell)$ is the modular curve associated to the normalizer of a non-split Cartan
subgroup of $\GL_2(\F_\ell)$, similar to the split case.
Mazur's method applies whenever $J_0(\ell) \ne 0$, which is the case for
$\ell=11$ and $\ell\ge 17$.
But since $J_0(13)=0$, it follows that $\mathrm{Jac}(X_{\sC}(13))\sim
\mathrm{Jac}(X_{\ns}(13))$ and $\mathrm{Jac}(X_{\sC}(13))$ is absolutely simple, which is the
underlying reason that their analysis does not
succeed in tackling that case; 
they call~13 the \textit{cursed level} in~\cite[Remark~5.11]{BPR13}.

\par In fact, Baran~\cite{Bar14a, Bar14b} showed that more is true: There is a $\Q$-isomorphism between
$\mathrm{Jac}(X_{\sC}(13))$ and $\mathrm{Jac}(X_{\ns}(13))$, and we further have 
\begin{equation}\label{eqn:baran_isom}
  X_{\ns}(13) \simeq_{\Q} X_{\sC}(13).
\end{equation}
She derives~\eqref{eqn:baran_isom} in two different ways: by computing explicit 
smooth plane quartic equations for both curves and observing that they are
isomorphic~\cite{Bar14a} on the one hand, and by invoking Torelli's theorem~\cite{Bar14b}
and the isomorphism between the Jacobians on the other.
There is no known modular interpretation of the isomorphism~\eqref{eqn:baran_isom}.
Since the problem of computing rational points on modular curves associated to
normalizers of non-split Cartan subgroups is believed to be hard in general, this
gives some indication why $X_{\sC}(13)$ is more difficult to handle than $X_{\sC}(\ell)$ for
other $\ell\ge 11$.

\par 
Galbraith~\cite{Gal02} and Baran\cite{Bar14a} computed all rational points up to a large
height bound; they found~6 CM-points and one cusp.
However, in addition to Mazur's method, it turns out that the other standard 
approaches to proving that this is the complete set of rational points also do not
seem to work for $X_{\sC}(13)$. 
The method of Chabauty and Coleman (see~\S\ref{chab-col}) fails because 
the rank of $\mathrm{Jac}(X_{\sC}(13))$ is at least~3, and the genus of $X_{\sC}(13)$ is 3. 
The Mordell-Weil sieve cannot
be applied on its own, because $X_{\sC}(13)(\Q) \neq \emptyset$. 
Descent techniques and elliptic curve Chabauty also do not seem to work in practice, because 
no suitable covers of $X_{\sC}(13)$ are readily available.

\par In this paper we show, using quadratic Chabauty,
that the only rational points on $X_{\sC}(13)$ are indeed the points
found by Galbraith and Baran.
\begin{theorem}\label{thm:main_thm}
The rational points on $X_{\sC}(13)$ consist of six CM-points and one cusp.
\end{theorem}

\par
Together with the results of Bilu--Parent and Bilu--Parent--Rebolledo,
this allows us to complete the characterisation
of all primes $\ell$ such that the mod-$\ell$ Galois representation of a non-CM elliptic curve over $\Q$
is contained in the normalizer of a split Cartan subgroup of $\GL_2(\F_\ell)$.
\begin{theorem}\label{thm:serre_split}
  Let $\ell$ be a prime number. Then there exists an elliptic curve $E/\Q$ without CM
  such that the image of its mod-$\ell$ Galois representation is contained in
   the normalizer of a split Cartan subgroup of $\GL_2(\F_\ell)$ if and only if $\ell \le 7$.
\end{theorem}

\par
Via the isomorphism~\eqref{eqn:baran_isom}, we also find
\begin{corollary}\label{cor:non_split}
  We have $| X_{\ns}(13)(\Q)| = 7$, and all points are CM.
\end{corollary}

\begin{Remark}\label{rk:class_number}
  As was noted by Serre \cite{Ser97} a complete determination of the rational points on
  $X_{\ns}(N)$ for some $N$ leads to a proof of the class number one problem.
  Corollary~\ref{cor:non_split} therefore gives a new proof of this theorem.
\end{Remark}

\subsection{Notation } \label{sec:notn}
\par
  Throughout this paper, $X/\Q$ will denote a smooth projective geometrically connected curve
  of genus $g \geq 2$ such that $X(\Q) \ne \emptyset$, and we denote its
  Jacobian $\mathrm{Jac}(X)$ by $\JacX$; 
  we write $r := \rk(\JacX/\Q)$ and $\rho := \rk(\NS(\JacX))$.
  We fix an algebraic closure $\overline{\Q}$ of $\Q$ and write 
  $G_{\Q} := \Gal(\overline{\Q}/\Q)$ and $\overline{X} := X \times \overline{\Q}$.
We fix a base point $b \in X(\Q)$ and a  prime $p$ such that $X$ has good reduction at $p$.
The field $\End(J)\otimes \Q$ is denoted by $K$ and we set 
$
\mathcal{E} :=  \HH^0 (X_{\Q _p },\Omega ^1 )^* \otimes \HH^0 (X_{\Q _p },\Omega ^1 )^*.
$ 
Denoting by $T_0$ the set of primes of bad reduction of $X$, we set $T = T_0 \cup \{p \}$,
and let $G_T$ denote the maximal quotient of $\Q$ unramified outside $T$. For a prime $v$, we let $G_v$ denote the absolute Galois group of $\Q_v$.

\subsection{Chabauty--Coleman and Chabauty--Kim } \label{chab-col}
\par Chabauty~\cite{Cha41} proved the Mordell conjecture for curves $X$ as above, satisfying an additional assumption on the rank of the Jacobian. More precisely, Chabauty showed that the set $X(\Q)$ is finite if $r < g$. Following Coleman ~\cite{Col85}, one may explain the proof as follows. The choice of base point $b$ gives an inclusion of $X$ into $\JacX$,
defined over $\Q$. On $\JacX(\Q _p )$ there is a linear integration pairing on the Jacobian defined by explicit power series integration on individual residue polydisks, extended via the group law
\[ \JacX(\Q_p) \times \HH^0(\JacX_{\Q_p},\Omega^1) \ \lra \ \Q_p: (D,\omega) \mapsto \int_{0}^D\omega,\]
inducing a homomorphism 
\[\log:\JacX(\Q_p) \ \lra \ \HH^0(\JacX_{\Q_p},\Omega^1)^\ast.\]
Via the canonical identification of $\HH^0(\JacX_{\Q_p},\Omega^1)$ with
$\HH^0(X_{\Q_p},\Omega^1)$, this gives rise to the following commutative diagram:
\begin{equation}\label{eqn:diagram-ab}
\resizebox{8.5cm}{!}{
\begin{tikzpicture}[->,>=stealth',baseline=(current  bounding  box.center)]
 \node[] (X) {$X(\Q)$};
 \node[right of=X, node distance=3.7cm]  (Xp) {$X(\Q_p)$};   
 \node[below of=X, node distance=1.5cm]  (Hf) {$\JacX(\Q)$};
 \node[right of=Hf,node distance=3.7cm] (Hfp) {$\JacX(\Q_p)$};
 \node[right of=Hfp,node distance=3.5cm](Dieu) {$\HH^0(X_{\Q_p},\Omega^1)^\ast$}; 

 \path (X)  edge node[left]{\footnotesize $$} (Hf);
 \path (Xp) edge node[left]{\footnotesize $$} (Hfp);
 \path (X)  edge (Xp);
 \path (Hf) edge node[above]{} (Hfp);
 \path (Hfp) edge node[above]{$\log $}(Dieu);
 \path (Xp)  edge node[above right]{$\AJb $} (Dieu);
\end{tikzpicture}}
\end{equation}
where the Abel-Jacobi morphism $\AJb $ is defined to be the
map sending a point $x$ to the linear functional $\omega \mapsto \int^x_b
\omega $. Chabauty's proof involves a combination of global ``arithmetic'' or ``motivic''
information with local ``analytic'' information. The global 
arithmetic input is that, when $r<g$, the closure $\overline{\JacX(\Q)}$ of $\JacX(\Q)$ with
respect to the $p$-adic topology is of codimension $\geq 1$. 
Hence there is a nonzero functional
$\omega_{J}$ which vanishes on $\overline{\JacX(\Q)}$, so that $X(\Q)$ is annihilated by the function
\begin{equation}\label{eqn:fn_chabauty}
  x \mapsto \AJb(x)(\omega_J). 
\end{equation}
The local analytic input is that, on each residue disk of $X(\Q _p )$, $\AJb$ has Zariski
dense image and is given by a convergent $p$-adic power series, so the function in~\eqref{eqn:fn_chabauty} can have only finitely many zeroes on each residue disk of $X(\Q _p )$. The non-trivial steps in solving for the function $\int _b \omega _J $ annihilating rational points are then:
\begin{itemize}
\item Determine, on each residue disk, the power series $\AJb$ to sufficient $p$-adic accuracy.
\item Evaluate $\AJb (P_i)$ on a basis $\{P_i\}$ of $\JacX(\Q )\otimes \Q$.
\end{itemize}

\par
With the aim of removing the restrictive condition $r<g$, Kim~\cite{Kim05,Kim09} has initiated a programme to
generalise Chabauty's approach. As in the method of Chabauty and Coleman, one hopes to be able to translate Kim's approach into a practical explicit method for computing (a finite set of $p$-adic points containing)
$X(\Q)$ in practice for a given curve $X/\Q$ having $r \ge g$. However, in part due to the technical nature of the objects involved, this is a rather delicate task. Kim's results~\cite{Kim05} on integral points on $\mathbb{P}^1\setminus\{0,1,\infty\}$ have been made explicit by Dan-Cohen and Wewers~\cite{DCW15} and used to develop an algorithm to solve the $S$-unit equation~\cite{DCW16, DC} using iterated $p$-adic integrals.
The work~\cite{BDCKW} of the first author with Dan-Cohen, Kim and Wewers contains explicit
results for integral points on elliptic curves of ranks 0 and 1.

\subsection{Quadratic Chabauty.} \label{sec:qc_pair}
\par
One approach that has led to some explicit results is to 
relate Kim's ideas to $p$-adic heights. 
Here we formalize this approach in elementary terms.
Suppose $r=g$, and the $p$-adic closure of
$\JacX(\Q )$ has finite index in $\JacX(\Q _p )$. Then the Abel-Jacobi morphism induces an isomorphism $\JacX(\Q)\otimes \Q _p \simeq \HH^0 (X_{\Q _p },\Omega ^1 )^* $, meaning that we cannot detect
global points among local points using linear relations in the Abel-Jacobi map. 
The idea of the quadratic Chabauty method is to replace linear relations by bilinear
relations. Suppose we can find a function
$
\theta :X(\Q _p )\to \Q _p 
$
and a finite set $\Upsilon \subset \Q _p $ 
with the following properties:
\begin{itemize}
\item[(a)] On each residue disk $]x[ \subset X(\Q _p )$, the map
\[
(\AJb ,\theta )\ :\ X(\Q _p )\ \lra \ \HH^0 (X_{\Q _p },\Omega ^1 )^* \times \Q _p
\] 
has Zariski dense image and is given by a convergent power series.
\item[(b)] There exist an endomorphism $E$ of $\HH^0 (X_{\Q _p },\Omega ^1 )^*$, a functional
  $c\in \HH^0 (X_{\Q _p },\Omega ^1 )^* $, and a bilinear form 
    $B : \HH^0 (X_{\Q _p },\Omega ^1)^* \otimes \HH^0 (X_{\Q_p},\Omega ^1 )^* \to \Q _p $,
    such that, for all $x\in X(\Q )$,
\begin{equation}\label{quadratic_chab}
\theta (x)-B(\AJb (x),E(\AJb (x))+c) \in \Upsilon .
\end{equation}
\end{itemize}
This gives a finite set of $p$-adic points containing $X(\Q )$, since property (a) implies that 
only finitely many $p$-adic points can satisfy equation \eqref{quadratic_chab}, and property (b) implies all rational points satisfy it.
As in the Chabauty-Coleman method, finiteness is obtained by a combination of local
analytic information and global arithmetic information. We shall refer to $(\theta,\Upsilon )$ as a \textit{quadratic Chabauty pair}. The objects $E,c,$ and $B$ of a quadratic Chabauty pair will be referred to as its endomorphism, constant and pairing, respectively. 

\par The goal of the quadratic Chabauty method is to be able to use a quadratic Chabauty pair (or several of them) to \textit{determine} $X(\Q )$. Let us clarify how the pair $(\theta ,\Upsilon )$ (as well as knowledge of the implicit $E$ and $c$) gives a method for determining a finite set containing $X(\Q )$.  
In practice, one can calculate $E$ and $c$, but $B$ is something one has to solve for, in the same way that one solves for the annihilating differential in the Chabauty-Coleman method. 

\par
For $\alpha \in \Upsilon $, define 
\[
  X(\Q _p )_{\alpha }:=\{x\in X(\Q _p ):\theta (x)-B(\AJb (x),E(\AJb (z))+c) = \alpha \}.
\] 
By definition, $X(\Q )\subset \sqcup _{\alpha \in \Upsilon }X(\Q _p )_{\alpha }$, and we focus on the problem of describing $X(\Q _p )_{\alpha }$. The following result gives an explicit equation for a finite subset of $X(\Q_p)$ containing $X(\Q_p)_{\alpha}$. Suppose we have $P_1 ,\ldots ,P_m\in X(\Q)$ such that 
\[
\AJb (P_i )\otimes (E(\AJb (P_i ))+c)
\]
form a basis of $\mathcal{E}$, and suppose that $\psi_1,\ldots,\psi_m$ form a basis of $\mathcal{E}^*$. Assume furthermore that we have $P_i \in X(\Q_p )_{\alpha _i }$, where $\alpha_i \in \Upsilon$. For $x \in X(\Q_p)$, define the matrix $T(x) = T_{(\theta, \Upsilon)}(x)$ by
\[
  T (x) = 
\left(
\begin{array}{cccc}
  \theta (x)-\alpha & \Psi_1 (x) & \ldots &   \Psi_m (x)   \\
  \theta (P_1 )-\alpha _1 &  \Psi_1 (P_1) & \ldots &  \Psi_m (P_1) \\
\vdots & \vdots & \ddots & \vdots \\
  \theta (P_m )-\alpha _m &   \Psi_1 (P_m) & \ldots &  \Psi_m (P_m) \\
 \end{array}
\right), 
\]
where $\Psi_i(x) :=  \psi_i (\AJb (x )\otimes (E(\AJb (x ))+c))$. Since $B$ is a linear combination of the $\psi_i$, we get:
\begin{lemma}\label{lemma:qc_pair}
  If $x \in X(\Q_p)_{\alpha }$, then we have
  \[
\det(T (x))=0.
  \]
\end{lemma}

\subsection{Quadratic Chabauty pairs for rational points}
\par
The definition of quadratic Chabauty pairs is inspired by an approach for computing
\textit{integral} points on rank 1 elliptic curves~\cite{BB15}, and more generally, on
odd degree hyperelliptic curves~\cite{BBM16}, which satisfy the assumptions of~\S\ref{sec:qc_pair}, as follows. 
Let $h:\JacX(\Q) \to \Q_p$ denote the $p$-adic height function~\cite{CG89}. 
Then, for $x \in X(\Q)$ there is a decomposition 
\begin{equation}\label{qc0_decomp}
  h(x-\infty) = h_p(x) +\sum_{v \ne p}h_v(x)
\end{equation}
of $h(x-\infty)$ into a sum of local heights
such that $x\mapsto h_p(x)$ extends to a locally
analytic function $\theta:X(\Q_p) \to \Q_p$ (in fact a sum of double Coleman integrals), and for $v \ne p$ the
function $x \mapsto h_v(x)$ 
maps integral points in $X(\Q)$ into a finite subset of $\Q_p$, and
this set is trivial if $v$ is a prime of good reduction.
By assumption, the $p$-adic height can be expressed
in terms of a bilinear map on $\HH^0(X_{\Q_p}, \Omega^1)^*$.
Because $\theta$ and the set $\Upsilon$ of possible values of $\sum_{v \ne p} h_v(x)$
for integral $x \in X(\Q)$ can be computed explicitly for a given curve $X$, 
this can be turned into a practical method for computing the integral points~\cite{BBM17}.
The main problem in generalizing this idea to rational points is that we have no 
way to control the values of the local heights $h_v(x)$ away from $p$ when $x$ is
rational, but not necessarily integral. 

\par 
Following ~\cite{BD16}, we construct a quadratic Chabauty pair by
associating to points of $X$ (a mixed extension of) Galois representations, and then 
taking the $p$-adic height of this Galois representation in the sense of Nekov\'a\v r ~\cite{Nek93}.
In~\cite[\S5]{BD16}, a suitable $G_L$-representation $A_Z(b,x)$ is constructed for every $x \in
X(L)$, where $L/\Q $. 
This depends on the choice of a correspondence $Z$ on $X$ satisfying certain properties;
such a correspondence always exists when $\rho>1$. By \cite[Theorem 1.2]{BD16}, the height of $A_Z (b,x)$ is equal to the height pairing between two divisors given explicitly in terms of $b,x$ and $Z$. An alternative approach, taken in this paper, is to work with the representation $A_Z (b,x)$ directly, without determining the corresponding divisors. The advantage of the latter approach is that one does not need an explicit geometric description of the correspondence $Z$, but only its cycle class.

\par
In the following, we shall denote Nekov\'a\v r's $p$-adic height by $h$.
Similar to~\eqref{qc0_decomp}, there is a local decomposition
\[
  h(A_Z(b,x)) = h_p(A_Z(b,x)) + \sum_{v \neq p} h_v(A_Z(b,x)),
\]
where now $x \mapsto h_p(A_Z(b,x))$ again extends to a locally analytic function 
$\theta:X(\Q_p)\to \Q_p$ by Nekov\'a\v r's construction~\cite{Nek93}, and for $v \ne p$ the local heights $h_v(A_Z(b,x))$ take on a finite set of values $\Upsilon$ for $x \in X(\Q_v)$ by a result of Kim and Tamagawa~\cite{KT08}.
According to~\cite[\S5]{BD16}, this gives a quadratic Chabauty pair $(\theta,\Upsilon)$ whose pairing is
$h$ and whose endomorphism is the one induced by $Z$. 

\par 
Suppose that $X$ satisfies $r=g$ and $\rho>1$, and that the $p$-adic closure of
$\JacX(\Q )$ has finite index in $\JacX(\Q _p )$.
Note that these conditions are satisfied for
many modular curves for which Chabauty-Coleman does not apply 
(see~\cite{Sik}), in particular for $X=X_{\sC}(13)$.  
Also suppose that we have enough rational points $P_1,\ldots,P_m$ to generate
$\mathcal{E}$ as in~\S\ref{sec:qc_pair}.
It follows from Lemma~\ref{lemma:qc_pair} that, if we can solve the following problems explicitly,
then we have an explicit method for computing a finite subset of $X(\Q_p)$
containing $X(\Q)$:
\begin{enumerate}[(i)]
  \item \label{away} Determine the set of values that $h_v(A_Z(b,x))$ can take for $x \in X(\Q_v)$ and $v
  \ne p$.
  \item \label{expand} Expand the function $x \mapsto h_p(A_Z(b,x))$ into a $p$-adic power series on every residue disk.
  \item \label{eval} Evaluate $h(A_Z (b,P_i))$ for $i=1,\ldots,m$.
\end{enumerate}

\par In this paper, we say nothing about problem~\eqref{away} since $X_{\sC}(13)$, our main
object of interest, has potentially good reduction everywhere, so that all local heights
away from $p$ are trivial. This also reduces problem~\eqref{eval} to problem~\eqref{expand}. Nevertheless, in the interest of future applications, we  phrase much of the setup in greater generality than needed for the application to $X_{\sC}(13)$. 

\subsection{Explicit local $p$-adic heights at $p$}
\par
  The main contribution of this paper is to give an explicit algorithm for solving
  problem~\eqref{expand}. This is already done for hyperelliptic curves in~\cite{BD17}, and
we follow the general strategy used there.
As in~\cite{Kim09, H11}, we emphasize the central role played by 
universal objects in neutral unipotent Tannakian categories. 
This approach allows us to make several aspects of~\cite{BD16} and~\cite{BD17} explicit in a
conceptual way. 

\par
The definition of Nekov\'a\v r's local height at $p$ is in terms of $p$-adic Hodge theory.
More precisely, let $M(x)$ denote the image of $A_Z(b,x)$ under Fontaine's
$\D_{\cris}$-functor.
Then $M(x)$ is a filtered $\phi$-module,
and by Nekov\'a\v r's definition,
to construct $h_p(A_Z(b,x))$, it suffices to explicitly describe its Hodge filtration and its Frobenius action, 
It is shown in~\cite{BD17} that $M(x)$ can be described as the pullback along $x$ (viewed as
a section of a suitable affine $\mathcal{Y}/\Z_p$) of a certain universal connection
$\mathcal{A}_Z$, which has the structure of a filtered $F$-isocrystal, and our task
is to find a sufficiently explicit description of this structure. 
In other words, we need to find both the Hodge filtration and the Frobenius
structure on $\mathcal{A}_Z$.
In~\cite{BD17}, the Hodge filtration is computed using a universal property proved by
Hadian~\cite{H11}, and we follow a similar strategy here.
In contrast, the explicit description of the Frobenius structure is more
involved, and constitutes the key new result which makes our approach work. 
In the hyperelliptic situation, one gets a description in terms of Coleman
integrals, but this crucially relies on the hyperelliptic involution~\cite[\S6.6]{BD17}.
Here we characterise the Frobenius structure using a universal property, based on work of
Kim~\cite{Kim09}.

\subsection{Algorithmic remarks and applicability}
\par
We note that while many of the constructions in this paper rely on deep results in
$p$-adic Hodge theory, for a given curve, all of this can subsequently be translated into
rather concrete linear algebra data which can be computed explicitly. For instance,  
instead of working with a correspondence $Z$ explicitly, by the $p$-adic Lefschetz (1,1) 
theorem it is enough to work with the induced Tate class in $\HH^1_{\dR}(X_{\Q_p})\otimes \HH^1_{\dR}(X_{\Q_p})$.
In practice, we fix a basis of $\HH^1_{\dR}(X_{\Q_p})$ and encode our Tate classes as matrices 
with respect to this basis. Computing the structure of $M(x)$ as a filtered $\phi$-module boils down to computing two isomorphisms of $2g+2$-dimensional $\Q_p$-vector spaces
\[
  \Q_p\oplus \HH^1_{\dR}(X_{\Q_p})^*\oplus \Q_p(1) \simeq M(x),
\]
one of which respects the Hodge filtration, while the other one is Frobenius-equivariant. 
In practice, the universal properties discussed above can be described in terms of explicit $p$-adic
differential equations, which we solve using algorithms of the fourth author~\cite{Tui16,
Tui17}.  All of our algorithms have been implemented in the computer algebra system {\tt
Magma}~\cite{BCP97}. 

\par
  The results of this paper remain useful in somewhat less restrictive situations than
  the one considered above. For instance, 
  as noted above, the condition that the curve has potentially good reduction everywhere is
  only used to give a particularly simple solution to problem \eqref{away}
  (and~\eqref{eval}). 
  Also,~\cite[\S5.3]{BD16} discusses an approach to computing a finite set containing $X(\Q)$
  when $r>g$, but $r+1-\rho<g$, and is similar to the one used here.  
  For this approach one also needs to solve problem~\eqref{expand}, and our algorithm for its
  solution applies without change.

 \par 
  Moreover, recall that we have made the assumption that we have enough rational points
  available to span $\mathcal{E}$ as in~\S\ref{sec:qc_pair}.
  In practice, since $\rho >1$, the algebra $K:=\End (J) \otimes \Q$ will be strictly larger than $\Q$ and,
  following~\cite{BD17}, we can construct $h$ so that it is $K$-equivariant.
  This means we can replace $\mathcal{E}$ by $\HH^0
  (X_{\Q _p },\Omega ^1 )^* \otimes _{K\otimes \Q _p } \HH^0 (X_{\Q _p },\Omega ^1 )^*$
  in Lemma~\ref{lemma:qc_pair},
  which cuts down the number of rational points required. We use this for $X=X_{\sC}(13)$.
If we have an algorithm to explicitly compute the $p$-adic height pairing between rational points on the Jacobian,
and we have enough independent rational points on the Jacobian,
then we only need one rational point on $X$, to serve as our base point.


\subsection{Outline}
\par
In Section \ref{sec:chab-col}, we recall the salient points of Chabauty-Kim theory and in 
Section~\ref{sec:heights} we recall the definition of Nekov\'a\v r's $p$-adic heights
and how $p$-adic heights can be used to construct quadratic Chabauty pairs for rational
points.
Section~\ref{sec:comphodge} discusses the computation of the Hodge
filtration on a universal connection $\mathcal{A}_Z$, and Section \ref{sec:compfrob} contains a
recipe for computing the Frobenius structure on $\mathcal{A}_Z$.
Both of these rely on universal properties and can be used, by pullback, to determine the
structure of $A_Z(b,x)$ as a filtered $\phi$-module.
All of the aspects of the theory are then
computed explicitly for $X = X_{\sC}(13)$ in Section \ref{sec:Xs13}:
We first show that the rank of $\JacX(\Q)$ is exactly~3 and that
 $X$ has potentially good reduction. 
We then run our algorithm for the computation of the local $17$-adic height at $p=17$ for
two independent Tate classes coming from suitable
correspondences, leading to two quadratic Chabauty pairs.
As a consequence, we prove Theorem~\ref{thm:main_thm}.
The appendix contains a discussion of some concepts and results on unipotent neutral
Tannakian categories used throughout the paper.

\subsection*{Acknowledgements}
We are indebted to Minhyong Kim for proposing this project, and for his suggestions and encouragement.  Balakrishnan is supported in part by NSF grant DMS-1702196, the Clare Boothe Luce Professorship (Henry Luce Foundation), and Simons Foundation grant \#550023.  Tuitman is a Postdoctoral Researcher of the Fund for Scientific Research FWO - Vlaanderen. Vonk is supported by a CRM/ISM Postdoctoral Scholarship at McGill University. 

\section{Chabauty-Kim and correspondences}
\label{sec:chab-col}
In this section we briefly recall the main ideas in the non-abelian Chabauty method of Kim \cite{Kim09}.
We then recall some results from \cite{BD16} which can be used to prove the finiteness of the set of rational points under certain assumptions. None of the results in this section are new. 

\par In a letter to Faltings, Grothendieck proposed to study rational points on $X$
through the geometric \'etale fundamental group $\pi^{\et}_1(\overline{X},b)$ of $X$
with base point $b$.  More precisely, he conjectured that the map
\[ X(\Q) \longrightarrow
\HH^1\left(G_{\Q},\pi^{\et}_1(\overline{X},b)\right),\]
given by associating to $x \in X(\Q)$ the \'etale path torsor
$\pi^{\et}_1(\overline{X};b,x)$, should be an isomorphism. Unfortunately, there seems to a
lack of readily available extra structure on the target, which makes it difficult to study directly.
However, one can try instead to work with a suitable quotient of $\pi^{\et}_1(\overline{X},b)$, where ``suitable'' depends on the properties of the curve in question. Indeed, most techniques for studying $X(\Q)$ can be phrased in this language. Chabauty--Coleman, finite cover descent (see for instance~\cite{BS09}) and elliptic curve Chabauty~\cite{FW99,Bru03} rely on \textit{abelian} quotients, whereas Chabauty--Kim, discussed below, uses \textit{unipotent} quotients. Following~\cite{BD16} we will construct quadratic Chabauty pairs for a class of curves including $X_{\sC}(13)$ from the simplest non-abelian unipotent quotient when $r=g$ and $\rho>1$.

\subsection{The Chabauty-Kim method}\label{sec:sel_rat} Via the Bloch-Kato exact sequence, there is an isomorphism $\Hf (G_p ,V)\simeq V_{\dR} /\Fil^0 $, where $V_{\dR}:= \HH^1 _{\dR} (X_{\Q _p })^* $, viewed as a filtered vector space with the dual filtration to the Hodge filtration. By definition, we obtain an isomorphism $V_{\dR}/\Fil^0 \simeq \HH^0 (X_{\Q _p },\Omega ^1 )^* $. It follows from \cite[3.10.1]{BK90} that 
this gives a commutative diagram
\begin{equation}\label{eqn:diagram-ab-funky}
\resizebox{8.5cm}{!}{
\begin{tikzpicture}[->,>=stealth',baseline=(current  bounding  box.center)]
 \node[] (X) {$X(\Q)$};
 \node[right of=X, node distance=3.7cm]  (Xp) {$X(\Q_p)$};   
 \node[below of=X, node distance=1.5cm]  (Hf) {$J(\Q )$};
 \node[right of=Hf,node distance=3.7cm] (Hfp) {$J(\Q _p )$};
 \node[right of=Hfp,node distance=3.5cm](Dieu) {$\HH^0 (X_{\Q _p },\Omega ^1 )^*  $};  
 \node[below of=Hf, node distance=1.5cm]  (Hff) {$\Hf (G_T ,V)$};
 \node[right of=Hff,node distance=3.7cm] (Hfpf) {$\Hf (G_p ,V)$};
 \node[right of=Hfpf,node distance=3.5cm](Dieuf) {$V_{\dR}/\Fil^0$};
 \path (X)  edge node[left]{\footnotesize } (Hf);
 \path (Xp) edge node[left]{\footnotesize } (Hfp);
 \path (X)  edge (Xp);
 \path (Hf) edge node[above]{} (Hfp);
\path (Hf) edge node[left]{ $\kappa $ } (Hff);
\path (Hfp) edge node[left ]{ $\kappa _p $ } (Hfpf);
\path (Dieu) edge node[left]{$\simeq $} (Dieuf) ;
 \path (Hfp) edge node[above]{$\log $}(Dieu);
 \path (Hff) edge node[above]{$\loc _p $} (Hfpf) ;
\path (Hfpf) edge node[above]{$\simeq$} (Dieuf);
\path (Xp)  edge node[above right]{$\AJb $} (Dieu);
\end{tikzpicture}}
\end{equation}
extending the Chabauty diagram~\eqref{eqn:diagram-ab}, where $\kappa$ and $\kappa_p$ map a point to its Kummer class. The idea of the Chabauty-Kim method is essentially that, if we cut out the middle row of this diagram, we obtain something amenable to generalisation. Namely, for each $n$ we obtain:
\begin{equation}\label{eqn:diagram-nonab}
\begin{tikzpicture}[->,>=stealth',baseline=(current  bounding  box.center)]
 \node[] (X) {$X(\Q)$};
 \node[right of=X, node distance=3.7cm]  (Xp) {$X(\Q_p)$};   
  \node[below of=X, node distance=1.5cm]  (Hf) {$\Sel(\U_n)$}; 
 \node[right of=Hf,node distance=3.7cm] (Hfp) {$\Hf(G_p,\U_n^{\et})$};
 \node[right of=Hfp,node distance=3.5cm](Dieu) {$\U^{\dR}_n/ \Fil^0 $.};

 \path (X)  edge node[left]{\footnotesize $j_n^{\et}$} (Hf);
 \path (Xp) edge node[left]{\footnotesize $j_{n,p}^{\et}$} (Hfp);
 \path (X)  edge (Xp);
 \path (Hf) edge node[above]{\footnotesize $\mathrm{loc}_{n,p}$} (Hfp);
 \path (Hfp) edge node[above]{\footnotesize $\D$} (Dieu);
 \path (Xp)  edge node[above right]{\footnotesize $j_n^{\dR}$} (Dieu);
\end{tikzpicture}
\end{equation}
The objects in this diagram will be explained in more detail in the next section. More generally, for any Galois stable quotient $\U$ of $\U_n ^{\et}$, we have a similar diagram involving $\U ^{\dR} :=\D_{\cris}(\U)$, and we will denote the corresponding maps by $j^{\et}_{\U},j^{\dR}_{\U}$ and $\loc_{\U,p}
$. We also define $\Sel(\U)$ exactly as for $\U_n^{\et}$. We then have that $X(\Q)$ is contained in the subset 
\[
  X(\Q_p)_{\U} := (j^{\et}_{\U,p})^{-1}\left(\loc_{\U,p} \Sel(\U)\right)  \subset X(\Q_p ).
\]
When $\U = \U_n$, we simply write $X(\Q_p)_{n}$ for this subset. We have
\[
    X(\Q)\subset\ldots \subset X(\Q_p)_n\subset X(\Q_p)_{n-1}\subset \ldots\subset X(\Q_p)_2\subset
    X(\Q_p)_1\subset X(\Q_p).
\]

\subsection{Selmer varieties}\label{sec:sel_var}
The main reference for this subsection is~\cite{Kim09}, and more detailed definitions are collected in Appendix \ref{sec:Tannakian}. 

\par Let $\U_n ^{\et}:= \U_n ^{\et }(b)$ denote the maximal $n$-unipotent quotient of the $\Q _p $-\'etale fundamental group of $\overline{X}$ with base point $b$. This is a finite-dimensional unipotent group over $\Q _p $ with a continuous action of $\Gal (\overline{\Q }/\Q )$, which contains the maximal $n$-unipotent
pro-$p$ quotient of $\pi _1 ^{\et}(\overline{X},b)$ as a lattice. In this paper, we only need $n=1$ or $2$. We have $\U_1^{\et} = V := \HH^1 _{\et }(\overline{X},\Q _p )^* $, and $\U_2^{\et}$ is a central extension
\begin{equation}\label{eqn:U2ext}
1 \lra \Coker \left(\Q_p (1) \stackrel{\cup ^* }{\lra }\wedge ^2 V \right) \lra \U _2 \lra V \lra 1.
\end{equation}
We obtain likewise for any $x \in X(\Q)$ a path torsor $\U_n^{\et}(b,x)$, see \ref{subsec:background_lisse}. This gives rise to a map
\[
j^{\et}_n: X(\Q ) \ \lra \ \HH^1 (G_T,\U_n^{\et}), \ x \mapsto \U_n^{\et}(b,x),
\]
as well as local versions $j_{n,v}^{\et}$ for any finite place $v$. We obtain the commutative diagram
\[
\begin{tikzpicture}[->,>=stealth',baseline=(current  bounding  box.center)]
 \node[] (X) {$X(\Q)$};
 \node[right of=X, node distance=4.7cm]  (Xv) {$\prod_{v \in T}X(\Q_v)$};   
 \node[below of=X, node distance=1.7cm]  (Hf) {$\HH^1(G_T,\U_n^{\et})$};
 \node[right of=Hf,node distance=4.7cm] (Hfv) {$ \prod_{v \in T}\HH^1(G_v,\U_n^{\et}).$};
\path (X)  edge node[left]{\footnotesize $j_n^{\et}$} (Hf);
 \path (Xv) edge node[right]{\footnotesize $\prod j_{n,v}^{\et}$} (Hfv);
 \path (X)  edge (Xv);
 \path (Hf) edge node[above]{\footnotesize $\prod \mathrm{loc}_{n,v}$} (Hfv);
\end{tikzpicture}
\]

\par By~\cite{Ols11}, we have $j_{n,p}^{\et}(X(\Q_p)) \subset \Hf(G_p, \U_n^{\et})$, where
$\Hf(G_p, \U_n^{\et})$ is the subspace consisting of crystalline torsors~\cite[(3.7.2)]{BK90}. It is shown in~\cite{Kim05} that $\HH^1(G_p,\U_n^{\et})$ and $\HH^1(G_T,\U_n^{\et})$ are represented by algebraic varieties over $\Q_p$. By~\cite[p. 119]{Kim09}, $\Hf(G_p,\U_n^{\et})$ is represented by a subvariety of $\HH^1(G_p,\U_n^{\et})$, and the analogous statement holds for $\Hf(G_T,\U_n^{\et}) := \loc_{n,p}^{-1}\Hf(G_p, \U_n^{\et})$. 
Similar to classical Selmer groups, we add local conditions and define the \textit{Selmer variety} 
$\Sel(\U_n)$ to be the subvariety of $\Hf(G_T,\U_n^{\et})$ consisting of all classes
\[
c \in \bigcap_{v \in T_0}\loc^{-1}_{n,v}(j^{\et}_{n,v}(X(\Q_v)))
\]
whose projection to $\HH^1_f(G_T,V)$ lies in the image of $\JacX(\Q)\otimes \Q_p$, see
diagram~\eqref{eqn:diagram-ab-funky} above. See~\cite[Remark~2.3]{BD16} for a discussion how our definition relates to other definitions of Selmer varieties (and schemes) in the literature.

\begin{Remark}
As our main example $X=X_{\sC}(13)$ has potentially good reduction everywhere (see Corollary~\ref{cor:redn_13} below), the local conditions at $v \ne p$ are vacuous, so in this case the Selmer variety could be defined in a simpler way.
\end{Remark}

\par Finally, we define the objects on the right side of diagram \eqref{eqn:diagram-nonab}. Let $L$ be a field of characteristic zero. Deligne \cite[Section 10]{Del89} constructs the \textit{de Rham fundamental group}
\[
\pi_1^{\dR}(X_L ,b),
\]
a pro-unipotent group over $L$, defined as the Tannakian fundamental group of the category $\mathcal{C}^{\dR} (X_L )$ of unipotent vector bundles with flat connection on $X$ with respect to the fibre functor $b^* $, see \S \ref{subsec:background_connections}. Here and in what follows, when there is no risk of confusion, we drop the subscript $L$. Define $\U _n ^{\dR}(b)$ to be the maximal $n$-unipotent quotient of $\pi _1 ^{\dR} (X,b)$, along with path torsors $\U_n^{\dR}(b,x)$ for all $x \in X(L)$. These have the structure of filtered $\phi$-modules, as explained in \S \ref{subsec:nonabeliancomparison}.

\par Kim shows \cite{Kim09} that the isomorphism classes of $\U_n^{\dR}$-torsors in the category of filtered $\phi $-modules are naturally classified by the scheme $\U _n ^{\dR}/\Fil^0$. Hence, we get a tower of maps 
\[
j^{\dR}_n : X(\Q_p) \ \lra \ \U_n ^{\dR}/\Fil^0,\ x \mapsto \U_n^{\dR} (b,x).
\]

\subsection{Diophantine finiteness. } In \cite{Kim09}, Kim showed how the set-up of Chabauty's theorem may be generalised to diagram~\eqref{eqn:diagram-nonab}. The sets in the bottom row have the structure of $\Q _p $-points of algebraic varieties, in such a way that the morphisms $\loc _{n,p}$ and $\mathbf{D}$ are morphisms of schemes (and $\mathbf{D}$ is an isomorphism). The analogue of the analytic properties of $\AJb$ is the theorem that $j_n ^{\dR}$ has Zariski dense image~\cite[Theorem~1]{Kim09} and is given by a power series on each residue disk. The analogue of Chabauty's $r<g$ condition is non-density of the localisation map $\loc _{\U,p}$. As in the classical case, this gives the following theorem.
\begin{theorem}[Kim]\label{thm:Kim-finite}
  Suppose $\loc_{\U,p} $ is non-dominant. Then $X(\Q_p )_{\U}$ is finite.
\end{theorem}
Kim~\cite[\S3]{Kim09} showed that non-density of $\loc _{\U ,p}$
(and hence finiteness of $X(\Q _p )_{\U} $) is implied by various conjectures on the size of unramified Galois cohomology groups (for example by the Beilinson--Bloch--Kato conjectures) but is hard to prove unconditionally. One instance where the relevant Galois cohomology groups can be understood by Iwasawa theoretic methods is when the Jacobian of $X$ has CM. This was used by Coates and Kim \cite{CK10} to prove eventual finiteness of weakly global points. Recently, Ellenberg and Hast \cite{EH} prove, using similar techniques, that the class of curves admitting an \'etale cover all of whose twists have eventually finite sets of weakly global points includes all solvable Galois covers of $\PP^1$. In this article the Galois cohomological input needed is of a much more elementary nature. The following result was proved by Balakrishnan--Dogra \cite[Lemma~3]{BD16}.

\begin{lemma}[Balakrishnan--Dogra]\label{lemma:NS_fin}
  Suppose 
  \[
  r <g+\rho-1.
\] 
  Then $X(\Q_p)_2 $ is finite.
\end{lemma} 
The idea of the proof of this lemma is as follows. As the map $\loc_{2,p}$ is algebraic,
it suffices by Theorem \ref{thm:Kim-finite} to construct a quotient $\U$ of $\U _2 $ for
which $\dim \Hf (G_T , \U)<\dim \Hf (G_p ,\U )$, since $X(\Q_p)_2 \subset X(\Q_p)_{\U}$. We can push out~\eqref{eqn:U2ext} to construct a quotient $\U$ of $\U_2$ which is an extension
\[ 1 \lra \Q_p(1)^{\oplus (\rho-1)} \lra \U \lra V\lra 1. \]
Using the six-term exact sequence in nonabelian cohomology and some $p$-adic Hodge theory, one shows $\dim \Hf(G_T,\U) \leq r$, whereas $\dim \Hf(G_p ,\U) = g+\rho-1$.

\subsection{Quotients of fundamental groups via correspondences }\label{sec:quotients}
Lemma \ref{lemma:NS_fin}, as well as the results of \cite{CK10,EH} where finiteness is proved unconditionally in certain cases, say nothing about how to actually {\em determine} $X(\Q_p)_{2}$ or $X(\Q)$ in practice. In~\cite{BD16, BD17}, the two first-named authors construct a suitable intermediate quotient $\U$ between $\U_2$ and $V$ that is non-abelian, but small enough to make explicit computations possible. 
  Working with such quotients $\U$, rather than directly with $\U_2$, may be thought of as a
non-abelian analogue of passing to a nice quotient of the Jacobian, 
as, for instance, in the work of Mazur~\cite{Maz77} and Merel~\cite{Mer96}. 
Lemma \ref{lemma:NS_fin} was deduced
from the finiteness of such a set $X(\Q_p)_{\U}$, and it is these sets which will be
computed explicitly in what follows. 

\par Denoting by $\tau$ the canonical involution $(x_1,x_2) \mapsto (x_2, x_1)$ on $X \times X$, we say that a
correspondence $Z \in \Pic(X \times X)$ is \textit{symmetric} if there are $Z_1, Z_2 \in \Pic(X)$ such that 
\[
\tau_\ast Z = Z + \pi_1^\ast Z_1 + \pi_2^\ast Z_2,
\]
where $\pi_1, \pi_2$ are the canonical projections $X \times X \to X$.

\par Now choose any Weil cohomology theory $\HH^*(X)$ with coefficient field $L$ of characteristic
zero. Via the cycle class map, a correspondence $Z$ on $X$ gives rise to a class in $\HH^2(X\times
X)(1)$, and hence a class $\xi_Z \in \HH^1(X) \otimes \HH^1(X)(1) \simeq \End \HH^1(X)$ via the K\"unneth decomposition. 
We say that $Z$ is a \textit{\nice} 
correspondence if $Z$ is nontrivial and symmetric and if $\xi_Z$ has trace~0. 
\begin{lemma}\label{lemma:kunneth} Suppose that $\JacX$ is absolutely simple 
 and let $Z \in \Pic(X \times
  X)$ be a symmetric correspondence.
  Then the class associated to $Z$ lies in the subspace 
\[\wedge^2 \HH^1(X) (1) \subset \HH^1(X) \otimes \HH^1(X)(1).\] 
Moreover, $Z$ is \nice if and only if 
the image of this class in $\HH^2(X)(1)$ under the cup product is zero.
\begin{proof}
  A correspondence $Z$ is symmetric if and only if the endomorphism of $\JacX$ induced by
  $Z$ is fixed by the Rosati involution. This follows from~\cite[Proposition~11.5.3]{BL04}
  (although the
  statement there assumes that $X$ is defined over the complex numbers, the proof works
  over any ground field).
  Since by~\cite[\S IV.20]{Mum70} the subspace of elements of $\End(\JacX)\otimes \Q$ fixed by the Rosati involution is isomorphic to $\NS(\JacX)\otimes \Q$, we find that $Z$ induces an element of $\NS(\JacX)$. Hence we obtain that the class associated to $Z$ lies in $\HH^2(\JacX) (1) = \wedge ^2 \HH^1(X) (1).$

The second statement is a consequence of the observation that the trace of $\xi_Z$ as a
linear operator on $\HH^1(X)$ is equal to the composite of the cup product and the trace
isomorphism $$\HH^1(X) \otimes \HH^1(X)(1) \rightarrow \HH^2(X)(1) \simeq L.$$ 
\end{proof}
\end{lemma}

\par We now define quotients $\U_Z$ of $\U_2$ attached to the choice of a \nice
correspondence $Z$ on $X$. These are morally responsible for the proof of Lemma
\ref{lemma:NS_fin}, and play a crucial role in our effective determination of $X_{\sC}(13)(\Q)$ 
By Lemma~\ref{lemma:kunneth}, if $Z$ is a \nice correspondence on $X$, 
we may take the pushout of \eqref{eqn:U2ext} by the dual of the Tate class
\[ \Q_p(-1) \lra \wedge^2 \HH^1_{\et}(X_{\overline{\Q}},\Q_p)\] 
associated to $Z$ to obtain an extension
  \[
1 \to \Q_p(1) \to \U _Z \to V \to 1.
\]

\begin{Remark}\label{remark:general_tate_class}
In the explicit computations of this paper, we will never work with \nice correspondences directly, but rather with their images in $\HH^1_{\dR}(X)\otimes\HH^1_{\dR}(X)(1)$.  
In fact, we can carry out these computations for quotients corresponding in the same
way to more general Tate  classes  $\HH^1_{\dR}(X)\otimes\HH^1_{\dR}(X)$ which come
from a \nice $Z \in \Pic(X\times X)\otimes
\Q_p$, for which we extend the notion of a \nice correspondence in the obvious way.
For notational convenience, we shall denote a Tate class obtained in this way by $Z$ as
well.
\end{Remark}


\section{Height functions on the Selmer variety}\label{sec:heights}

In this section we recall Nekov\'a\v r's theory of $p$-adic height functions
\cite{Nek93} and summarise some results of \cite{BD16} relating 
$p$-adic heights to Selmer varieties and leading to a construction of quadratic
Chabauty pairs when $r=g$ and $\rho>1$.

\subsection{Nekov\'a\v r's $p$-adic height functions. }We start by recalling some definitions from the theory of $p$-adic heights due to Nekov\'a\v r \cite{Nek93}. For a wide class of $p$-adic Galois representations $V$, Nekov\'a\v r \cite[Section 2]{Nek93} constructs a continuous bilinear pairing
\begin{equation}\label{eqn:Nekov\'a\v rGlobal}
h: \Hf(G_T,V) \times \Hf(G_T,V^*(1)) \lra \Q_p.
\end{equation}
This global height pairing depends only on the choice of
\begin{itemize}
\item a continuous id\`ele class character $\chi: \mathbf{A}_{\Q}^\times / \Q^{\times} \lra \Q_p$,
\item a splitting $s: V_{\dR}/\Fil^0 V_{\dR} \lra V_{\dR}$ of the Hodge
  filtration, where $V_{\dR} = \D_{\cris}(V)$. 
\end{itemize}
Henceforth, we fix such choices once and for all. We will only consider $V =
\HH^1_{\et}(X_{\overline{\Q}},\Q_p)^*$, and specialise immediately to this case for
simplicity, so that $V_{\dR} = \HH^1 _{\dR} (X_{\Q _p })^*$.

\par The global $p$-adic height pairing $h$ decomposes as the sum of local pairings $h_v$, for every non-archimedean place $v$ of $\Q$, as explained in \cite[Section 4]{Nek93}. As in the classical decomposition of the height pairing, the local height functions do not define a bilinear pairing, but rather a bi-additive function on a set of equivalence classes of mixed extensions, which we now explain. In the particular example of $X = X_{\mathrm{s}}(13)$, only the local height $h_p$ is of importance. Recall that $T = T_0 \cup \{p \}$, where $T_0$ is the set of primes of bad reduction of $X$.

\begin{Definition}\label{def:mixed-ext} Let $G$ be the Galois group $G_T$ or $G_v$, for $v \in T$. Define $\mathcal{M}(G;\Q_p ,V,\Q_p (1))$ to be the category of \textit{mixed extensions} with graded pieces $\Q _p ,V,$ and $\Q_p (1)$, whose objects are tuples $\left(M,M_{\bullet},\psi_{\bullet}\right)$ where 
\begin{itemize}
\item $M$ is a $\Q_p $-representation of $G$,
\item $M_{\bullet}$ is a $G$-stable filtration
$M=M_0 \supset M_1 \supset M_2 \supset M_3 =0$,
\item $\psi _{\bullet} $ are isomorphisms of $G_*$-representations
\[ \left\{ \begin{array}{lll}
\psi _0 :M_0 /M_1 & \stackrel{\sim }{\lra } & \Q_p, \\
\psi _1 :M_1 /M_2 & \stackrel{\sim }{\lra } & V, \\
\psi _2 :M_2 /M_3 & \stackrel{\sim }{\lra } & \Q_p (1), \\
\end{array}\right.\]
\end{itemize}
and whose morphisms 
\[
(M,M_{\bullet},\psi_{\bullet}) \ \lra \ (M',M'_{\bullet},\psi_{\bullet}')
\]
are morphisms $M \to M'$ of representations which respect the filtrations and commute with the isomorphisms $\psi _i $ and $\psi _i '$. Let $\mathrm{M} (G;\Q_p ,V,\Q_p (1))$ denote the set of isomorphism classes of objects. 

\par When $G=G_T$ or $G_p $, we denote by $\cMf (G;\Q _p ,V,\Q _p (1))$ the full subcategory of $\mathcal{M}(G;\Q _p ,V,\Q _p (1))$ consisting of crystalline representations, and similarly define $\Mf (G;\Q _p ,V,\Q _p (1))$.
\end{Definition}

\par The set $\Mf(G_T ;\Q_p ,V,\Q_p (1))$ is equipped with two natural surjective homomorphisms
\[
\begin{array}{lllll}
\pi_1: & \Mf (G_T ;\Q_p ,V,\Q_p (1)) &\lra & \Hf (G_T ,V) \ , & M \mapsto [M/M_2],\\
\pi_2: & \Mf (G_T ;\Q_p ,V,\Q_p (1)) &\lra & \Ext^1 _{G_T ,f}(V,\Q_p (1)) \ , & M \mapsto [M_1]. \\ 
\end{array} 
\]
(and similarly for the  groups $G_v$, for $v \in T_0$, and $G_p$).
Throughout this paper, we implicitly identify $\Hf (G_T ,V)$ and $\Hf(G_T,V^*(1))$ with the groups $\Ext^1 _{G_T ,f}(\Q_p,V)$ and $\Ext^1 _{G_T ,f}(V,\Q_p (1))$ respectively, where the subscript $f$ denotes those extensions which are crystalline at $p$. Via Poincar\'e duality, we may view $\pi _1 (M)$ and $\pi _2 (M)$ as both being elements of $H^1 _f (G_T ,V)$. We say $M$ is a \textit{mixed extension of $\pi _1 (M)$ and $\pi _2 (M)$.}

\par Nekov\'a\v r's global height pairing \eqref{eqn:Nekov\'a\v rGlobal} decomposes as a sum of local heights in the following sense. There exists for every finite place $v\neq p$ a function
\[
h_v : \mathrm{M}(G_v ;\Q _p ,V,\Q_p (1)) \ \lra \ \Q_p
\]
and a function
\[
h_p : \Mf (G_p ;\Q _p ,V,\Q_p (1)) \ \lra \ \Q_p
\]
such that $h = \sum_v h_v$, where $h$ is viewed by abuse of notation as the composite function
\[ 
\Mf (G_T ;\Q _p ,V,\Q _p (1)) \xrightarrow{(\pi_1,\pi_2)} \Hf (G_T ,V)\times \Ext^1 _{G_T ,f}(V,\Q_p (1)) \stackrel{h}{\lra} \Q_p. 
\]
We note that unlike the global height $h$, the local heights $h_v$ do not factor through the map analogous to $(\pi _1,\pi _2)$. We now define the functions $h_v$ for $v\neq p$ and $v=p$, and refer to Nekov\'a\v r \cite[Section 4]{Nek93} for more details. 

\subsection{The local height away from $p$. }We briefly recall the definition of the local height away from $p$. The construction rests on the fact that, by the weight-monodromy conjecture for curves proved by Raynaud \cite{Ray94}, we have 
$
\HH^1(G_v ,V) = 0.
$
This implies that a mixed extension $M$ in $\Mf(G_v ;\Q _p ,V,\Q _p (1))$ splits as $M \simeq V \oplus N$, where $N$ is an extension of $\Q_p$ by $\Q_p(1)$. We obtain a class $[N] \in \HH^1(G_v,\Q_p(1))$. The local component $\chi_v: \Q_v^{\times} \lra \Q_p$ gives a map
\[ 
\chi_v : \HH^1(G_v,\Q_p(1)) \simeq \Q_v^{\times} \widehat{\otimes} \Q_p \lra \Q_p,
\]
where the isomorphism is provided by Kummer theory. The local height at $v$ is now defined as 
\[
h_v(M) := \chi_v([N]).
\] 
When $M$ is unramified at $v$, the local height automatically vanishes. The same is true when $V$ becomes unramified over a finite extension of $\Q_v$:
\begin{lemma}\label{lemma:height-away-p}
  Let $v \neq p$, and let $M \in \mathrm{M}(G_v ;\Q _p ,V,\Q _p (1))$ be a mixed extension. Assume that $M$ is potentially unramified, then $h_v(M) = 0$.
\begin{proof}
Assume that $K_v / \Q_v$ is a finite Galois extension such that the action of $G_{K_v}$ on $M$ is unramified. The inflation-restriction sequence attached to this subgroup gives an exact sequence
\[ 0 \lra \HH^1(A, \Q_p(1)^{G_{K_v}}) \lra \HH^1(G_v, \Q_p(1)) \stackrel{\scriptsize \mathrm{res}}{\lra} \HH^1(G_{K_v}, \Q_p(1)),\]
where $A = G_v / G_{K_v}$ is a finite group. Write $M \simeq V \oplus N$, where $N$ is an extension of $\Q_p$ by $\Q_p(1)$. Then by assumption, we have that the class of $N$ in $\HH^1(G_{K_v}, \Q_p(1))$ is trivial. On the other hand, the restriction map $\mathrm{res}$ is injective, since $\Q_p(1)^{G_{K_v}}=0$. This shows that the class of $N$ in $\HH^1(G_v, \Q_p(1))$ is trivial, and in particular that $h_v(M)=0$.
\end{proof}
\end{lemma}

\subsection{The local height at $v = p$. }
\label{subsec:local_height_p}
Given a mixed extension $M_{\et} \in \Mf(G_p ;\Q _p ,V,\Q _p (1))$, the definition of its local height at $p$ will be given in terms of its image $M_{\dR}$ under Fontaine's $\D_{\cris}$ functor. The module $M_{\dR}$ inherits a structure similar to that of $M_{\et}$, which we formalise in the following definition.

\begin{Definition}\label{def:fil_phi_mod}
A filtered $\phi $-module is a finite-dimensional $\Q _p $-vector space $W$ equipped with an exhaustive and separated decreasing filtration $\Fil^i$ and an automorphism $\phi =\phi _W $.
\end{Definition}

Really, we are only interested in \textit{admissible} filtered $\phi $-modules, but since we will only consider iterated extensions of filtered $\phi $-modules which are admissible, and any extension of two admissible filtered $\phi $-modules is admissible, we will ignore this distinction. 

\par For any filtered $\phi $-module $W$ for which $W^{\phi =1}=0$, we have 
\begin{equation}\label{eqn:exp-map}
\Ext ^1 _{\fil ,\phi }(\Q _p ,W) \simeq W/\Fil^0,
\end{equation}
see for instance \cite[\S 3.1]{Nek93}. Explicitly, the map from the Ext group to $W/\Fil^0 $ is as follows. Given an extension
\[
0 \lra W \lra E \lra \Q_p \lra 0 ,
\]
one chooses a splitting $s^\phi :\Q _p \to E $ which is $\phi $-equivariant, and a splitting $s^{\Fil}$ which respects the filtration. Their difference gives an element of $W$. Since $W^{\phi =1}=0$, the splitting $s^\phi $ is unique, whereas $s^{\Fil}$ is only determined up to an element of $\Fil^0 W$. Hence the element $s^\phi -s^{\Fil} \in W \mod \Fil^0 $ is independent of choices. We leave the construction of the inverse map to the reader. 

\begin{Definition}\label{def:mixed-ext-filphi}Let $V$ be as above, and $V_{\dR} = \D_{\cris}(V)$. Define $\mathcal{M}_{\fil,\phi}(\Q_p ,V,\Q_p (1))$ to be the category of \textit{mixed extensions of filtered $\phi$-modules}, whose objects are tuples $\left(M,M_{\bullet},\psi_{\bullet}\right)$ where 
\begin{itemize}
\item $M$ is a filtered $\phi$-module,
\item $M_{\bullet}$ is a filtration by sub-filtered $\phi$-modules
$M=M_0 \supset M_1 \supset M_2 \supset M_3 =0$,
\item $\psi_{\bullet}$ are isomorphisms of filtered $\phi$-modules
\[ \left\{ \begin{array}{lll}
\psi _0 :M_0 /M_1 & \stackrel{\sim }{\lra } & \Q_p \\
\psi _1 :M_1 /M_2 & \stackrel{\sim }{\lra } & V_{\dR} \\
\psi _2 :M_2 /M_3 & \stackrel{\sim }{\lra } & \Q_p (1) \\
\end{array}\right.\]
\end{itemize}
and whose morphisms are morphisms of filtered $\phi$-modules which in addition respect the filtrations $M_{\bullet}$ and commute with the isomorphisms $\psi_i $ and $\psi_i '$. Let $\mathrm{M}_{\fil,\phi}(\Q_p ,V,\Q_p (1))$ denote the set of isomorphism classes of objects. 
\end{Definition}

\par The structure of a mixed extension on $M_{\dR} = \D_{\cris}(M_{\et})$ naturally allows us to define extensions $E_1 (M)$ and $E_2 (M)$ by
\begin{equation}\label{E1E2}
  E_1 (M) := M_{\dR} /\Q_p (1) , \qquad E_2 (M) :=\Ker (M_{\dR} \to \Q_p ),
\end{equation}
compare with the definition of $\pi_1$ and $\pi_2$ above. For simplicity we will sometimes write these as $E_1$ and $E_2$. We have a short exact sequence
\begin{equation}\label{makinght}
0 \lra \Q _p (1) \lra E_2 /\Fil^0 \lra V_{\dR} /\Fil^0  \lra 0.
\end{equation}
The image of the extension class $[M] \in \Hf(G_p,E_2) \simeq E_2 /\Fil^0$ in the group $V_{\dR} /\Fil^0 \simeq \Hf(G_p,V_{\dR})$ is exactly the extension class $[E_1]$. We define $\delta $ to be the composite map
\[
\delta \ :\ V_{\dR} /\Fil^0 \stackrel{s}{\lra} V_{\dR} \lra E_2 \lra E_2 /\Fil^0,
\]
where the homomorphism $V_{\dR} \to E_2 $ is the unique Frobenius-equivariant splitting of 
\[
0 \lra \Q_p (1)\lra E_2 \lra V_{\dR} \lra 0,
\]
and the last map is just the quotient. By construction, $[M]$ and $\delta ([E_1])$ have the
same image in $V_{\dR}/\Fil^0$, hence via the exact sequence (\ref{makinght}) their
difference defines an element of $\Q_p (1)$. The filtered $\phi$-module $\Q_p(1)$ is
isomorphic to $\Hf (G_p ,\Q_p (1))$ via \eqref{eqn:exp-map}, so we may think of
$[M]-\delta ([E_1])$ as an element of $\Hf (G_p ,\Q_p (1))$. The local component $\chi_p: \Q_p^{\times} \lra \Q_p$ gives rise to a map
\[ 
\chi_p : \Hf(G_p, \Q_p(1)) \simeq \Z_p^{\times} \widehat{\otimes} \Q_p \lra \Q_p
\]
where the isomorphism follows from Kummer theory. This allows us to define
\begin{equation}\label{loc_nek_def}
h_p (M) := \chi_p \left([M]-\delta ([E_1])\right).
\end{equation}

\par For the practical determination of rational points, it will be necessary to make this more explicit. To do so, it is convenient to introduce some notation for filtered $\phi$-modules $M$ in $\mathrm{M}_{\fil ,\phi}(\Q _p ,V ,\Q_p(1))$. The splitting $s$ of $\Fil^0 V_{\dR}$ defines idempotents $s_1 ,s_2 :V_{\dR} \lra V_{\dR}$ projecting onto the $s(V_{\dR}/\Fil^0)$ and $\Fil^0$ components respectively. Suppose we are given a vector space splitting 
\[
s_0 : \Q_p \oplus V_{\dR} \oplus \Q_p(1) \stackrel{\sim }{\longrightarrow }M.
\]
The split mixed extension $\Q_p \oplus V_{\dR} \oplus \Q_p(1)$ has the structure of a filtered $\phi$-module as a direct sum. Choose two further splittings 
\[
\begin{array}{llll}
s^\phi &: \ \Q_p \oplus V_{\dR} \oplus \Q_p(1) & \stackrel{\sim }{\lra } & M \\
s^{\Fil} &: \ \Q_p \oplus V_{\dR} \oplus \Q_p(1) & \stackrel{\sim }{\lra } & M \\
\end{array}
\]
where $s^{\phi}$ is Frobenius equivariant, and $s^{\Fil}$ respects the filtrations. Note that the choice of $s^{\phi}$ is unique, whereas the choice of $s^{\Fil}$ is not. Suppose we have chosen bases for $\Q_p, V_{\dR}$, and $\Q_p(1)$ such that with respect to these bases, we have
\[
s_0 ^{-1} \circ s^\phi =\left(
\begin{array}{ccc} 
1 & 0 & 0 \\
\boldsymbol{\alpha }_{\phi} & 1 & 0 \\
\gamma_{\phi} & \boldsymbol{\beta }^{\intercal}_{\phi} & 1 \\
\end{array}
\right) \qquad 
s_0 ^{-1}\circ s^{\Fil} = \left( 
\begin{array}{ccc}
1 & 0 & 0 \\
0 & 1 & 0 \\
\gamma_{\Fil} & \boldsymbol{\beta }^{\intercal}_{\Fil}  & 1 \\
\end{array}
\right).
\]
Then, Nekov\'a\v r's local height at $p$ defined by \eqref{loc_nek_def} translates into our notation to the simple expression  
\begin{equation}\label{eqn:height-decomposition}
h _p (M)= \chi_p \left( \gamma_{\phi} - \gamma_{\Fil} -\boldsymbol{\beta}^{\intercal}_{\phi} \cdot s_1 (\boldsymbol{\alpha}_{\phi})-\boldsymbol{\beta }^{\intercal}_{\Fil} \cdot s_2 (\boldsymbol{\alpha}_{\phi}) \right).
\end{equation}

\subsection{Twisting and $p$-adic heights. }\label{sec:twist_ht}
We now explain how to use Nekov\'a\v r's theory of $p$-adic heights for the desired Diophantine application. The method is a slight variation of the approach outlined in \S\ref{sec:sel_rat} via diagram \eqref{eqn:diagram-nonab}. See~\cite[Section~5]{BD16} for further details on the twisting construction. 

\par First let us briefly clarify the strategy for constructing a quadratic Chabauty pair using $p$-adic heights. Suppose that $X$ satisfies the conditions of~\S\ref{sec:qc_pair}. In particular, we
have $r=g$ and $\Hf(G_T,V) \simeq \HH^0(X_{\Q_p}, \Omega^1)^*$, so we can view the global height $h$ as a bilinear pairing
 \[
 h : \HH^0(X_{\Q_p}, \Omega^1)^* \times \HH^0(X_{\Q_p}, \Omega^1)^* \lra \Q_p.
 \]
\begin{lemma}\label{lemma:stating_obvious}
Suppose that, for all $v\in T_0 $, we have functions 
\[
\rho _v :X(\Q _v ) \lra \mathrm{M} (G_v ;\Q _p ,V,\Q _p (1)),
\] 
as well as a function
\[
\rho _p :X(\Q _p ) \lra \Mf (G_p ;\Q _p ,V,\Q _p (1)),
\]
with the following properties:
\begin{enumerate}
\item[(i)] For all $v\in T_0 $, $\rho _v $ has finite image.
\item[(ii)]   There exist $E \in \End \HH^0 (X_{\Q _p },\Omega^1 )^*$ and $c\in \HH^0 (X_{\Q _p },\Omega ^1 )^*$ such that 
\[
(\pi _1 ,\pi _2 )\circ \rho _p =(\AJb ,E (\AJb) +c).
\]
 \item[(iii)]  For all $z\in X(\F _p)$, the map $\rho _p |_{]z[}$ is given by a power series, and the map 
\[
(\AJb ,h_p \circ \rho _p ):]z[ \ \lra \ \HH^0 (X_{\Q _p },\Omega ^1 )^* \times \Q _p
\]
has Zariski dense image.
\item[(iv)]    For all $x\in X(\Q )$, there is a mixed extension $\rho (x) \in \Mf (G_T ;\Q _p ,V,\Q _p (1))$ with the property that $\loc _v (\rho (x))=\rho _v (x)$ for all $v\in T$.
\end{enumerate}
Then $(h_p \circ \rho _p ,\Upsilon )$ is a quadratic Chabauty pair, where 
\[
\Upsilon :=\left\{ \sum _{v\in T_0} h_v (\rho _v (x_v )):(x_v )\in \prod _{v\in T_0 }X(\Q _v ) \right\},
\]
with endomorphism $E$ and constant $c$.
\begin{proof}
By the first property, $\Upsilon $ is finite. 
The analyticity properties hold by assumption. The last property implies that for all $x\in X(\Q )$, 
\[
h_p \circ \rho _p (x) -B(\AJb (x),E \AJb (x)+c) \in \Upsilon
\]
where $B=h$ is the global bilinear height pairing.
\end{proof}
\end{lemma}

\par Let $I$ denote the augmentation ideal of $\Q _p [\pi _1 ^{\et}(\overline{X},b)]$, and
define $A_n^{\et}(b):= \Q _p [\pi _1 ^{\et}(\overline{X},b)]/I^{n+1} $. Then $A_n
^{\et}(b)$ is a nilpotent algebra, and by the theory of Malcev completion \cite[Appendix A.3]{Qui69} the limit of the $A_n$ is isomorphic to the pro-universal enveloping algebra of the $\Q_p$-unipotent \'etale fundamental group of $\overline{X}$ at $b$. There is an isomorphism 
\[
I^2 /I^3 \simeq \Coker \left(\Q _p (1) \stackrel{\cup ^* }{\lra } V^{\otimes 2}\right).
\]

\par Fix a \nice correspondence $Z \in \Pic(X \times X)$ (or, more generally, in $\Pic(X \times X)\otimes \Q_p$, see Remark~\ref{remark:general_tate_class}), and let $\U_Z$ denote the corresponding quotient of $\U_2^{\et}$ as defined in \S\ref{sec:quotients}. We define the mixed extension $A_Z (b)$ to be the pushout of $A_2 (b)$ by 
\[
\cl _Z ^* : \Coker \left(\Q _p (1) \stackrel{\cup ^* }{\lra } V^{\otimes 2} \right) \ \lra \ \Q _p (1),
\]
see also \cite[\S 5]{BD16}. Then $A_Z(b)$ defines an element in $\cMf(G_T; \Q_p, V, \Q_p(1))$, with respect to the $I$-adic filtration. The mixed extension $A_Z (b)$ is naturally equipped with a faithful Galois-equivariant action of $\U_Z$ which acts unipotently with respect to the filtration. 

\par By twisting we obtain the maps
\[
\begin{array}{llcclll}
\tau & : \ \Hf (G_T ,\U_Z) & \lra & \Mf (G_{T}; \Q_p,V,\Q_p(1)), & P & \longmapsto &
 P\times_{\U_Z} A_Z(b), \\
  \tau_p & : \ \Hf (G_p ,\U_Z) & \lra & \Mf (G_p; \Q_p,V,\Q_p(1)), & P & \longmapsto &
 P\times_{\U_Z} A_Z(b).  \\
\end{array}
\]
As explained in \cite[\S5.1]{BD16}, $\tau$ is injective. When $P$ is a torsor for a group $G$, we will denote $M^{(P)}$ for the twist $P \times_G M$ of a $G$-module $M$. When $P=\pi _1 ^{\et}(\overline{X};b,x)$, we denote 
\[
A_Z (b,x) := \tau([P]) =  A_Z (b)^{(P)} 
\]
and
$A_1 (b,x) := A_1 (b)^{(P)}$. We similarly define $IA_Z (b,z):=IA_Z (b)^{(P)}$. Composing with the unipotent Kummer map, we obtain a function
\[ 
\theta \ :\ X(\Q _p ) \ \lra \ \Q _p \ ; \ x \ \longmapsto \ h_p \left( A_Z (b,x)\right).
\]

\par First, let us discuss the local heights $h_v$, when $v \in T_0$. It was shown in Kim--Tamagawa \cite[Corollary 0.2]{KT08} that the $v$-component of the unipotent Albanese map has finite image.
\begin{theorem}[Kim--Tamagawa] \label{thm:KT}
  The map $j^{\et}_{2,v}: X(\Q_v) \lra \HH^1(G_v, \U_2)$ has finite image. 
\end{theorem}
This theorem implies that the following set is finite:
\[ 
\Upsilon := \left\{ \sum _{v \in T_0} h_v (A_Z (b,x_v )):(x_v )\in \prod_{v \in T_0} X(\Q
_v ) \right\} \subset \Q _p.
\]

\begin{lemma}\label{lemma:ht_qcpair}
With these definitions, $\left(\theta, \Upsilon \right)$ is a quadratic Chabauty pair. The
endomorphism of the quadratic Chabauty pair is that induced by $Z$, and the constant is $[IA_Z (b)]$.
\begin{proof}
It is enough to show that the maps $\rho _v =[A_Z (b,.)]$ satisfy the hypotheses of Lemma
\ref{lemma:stating_obvious}. Finiteness of $\rho _v $ when $v\neq p$ follows from the
Theorem of Kim and Tamagawa, proving hypothesis (i).
As explained in \cite[\S 5.2]{BD16}, $A_Z (b,x)$ is a mixed extension of
$\kappa (x-b)$ and $E_Z (\kappa (x-b)) + [IA_Z (b)]$, where $E_Z$ is the endomorphism
induced by $Z$, which proves (ii). 
Recall that the map $j_{\U_Z,p}^{\et}$ has Zariski dense image. By \cite[Lemma 10]{BD17}, the map 
\[
\left( \pi_*,\ h_p \circ \tau_p \right) \ : \  \Hf (G_p ,\U_Z) \ \lra \ \Hf (G_p ,V) \times \Q _p 
\]
is an isomorphism of schemes. The third property holds, since 
\[
(\AJb ,h_p (A_Z(b,.))) = (\pi _* ,h_p \circ \tau_p) \circ j ^{\et} _{\U_Z,p}.
\]
To prove (iv), we take $\rho(x) = [A_Z(b,x)]$.
\end{proof}
\end{lemma}

\par By Lemma~\ref{lemma:height-away-p}, we have that the local heights away from $p$ are all trivial if $X$ has potentially good reduction everywhere, so that $\Upsilon=\{0\}$. We obtain the following corollary:
\begin{corollary}\label{cor:qc_pair_pg}
If $X$ has potential good reduction everywhere, then $\left(\theta, \{0\} \right)$ is a
quadratic Chabauty pair, where $\theta = h_p (A_Z (b,.))$. The endomorphism of the quadratic Chabauty pair is that induced by $Z$, and the constant is $[IA_Z (b)]$.
\end{corollary}

\begin{Remark}\label{rk:ht_equiv}
We say the splitting  $s$ of the Hodge filtration is \textit{$K$-equivariant} if it commutes with the action of $K$. If $s$ is a $K$-equivariant splitting, then by \cite[\S4.1]{BD17} the associated height pairing on any two extensions $E_1, E_2$ is $K$-equivariant, in the sense that
\[
h(\alpha E_1, E_2) = h(E_1, \alpha E_2)
\]
for all $\alpha \in K$. We can use this to decrease the number of rational points required
to determine $X(\Q)$ using quadratic Chabauty.
\end{Remark}

\begin{Remark} In the general theory of $p$-adic heights, the character $\chi$ is
  essential for giving the right way to combine the various local classes in $\HH^1 (G_v
  ,\Q_p (1))$ which the theory constructs. However, for the application to $X_{\mathrm{s}}(13)$,
  or any curve with potentially good reduction everywhere, the local heights away from $p$
  are all trivial by Lemma \ref{lemma:height-away-p}.
  This means that the only purpose of $\chi$ is to define an isomorphism 
\[
\Hf (G_p ,\Q_p (1)) \simeq \Q_p .
\]
For this reason we henceforth 
drop $\chi$ from our notation, and focus in the next section on describing the class
  $[M]-\delta ([E_1(M)])$ for the mixed extension $M = A_Z(b,x)$, where $x \in X(\Q)$. 
\end{Remark}

\section{Explicit computation of the $p$-adic height I: Hodge filtration}
\label{sec:comphodge}

To complete the recipe for finding explicit finite sets containing $X(\Q)$, it remains to choose a \nice class $Z \in \Pic(X \times X)\otimes \Q_p$, and write the resulting locally analytic function 
\[
\theta \ :\ X(\Q _p ) \ \lra \ \Q _p \ ; \ x \ \longmapsto \ h_p \left( A_Z (b,x)\right)
 \] 
as a power series on every residue disk of $X(\Q_p)$. By equation \eqref{eqn:height-decomposition}, all that is needed is a sufficiently explicit description of the filtered $\phi$-module 
\[
\D_{\cris}(A_Z (b,x)).
\]
We compute the filtration and the Frobenius separately, as pull-backs of certain universal objects $\cA _Z ^{\dR}$ and $\cA _Z ^{\rig }$ respectively. The filtered structure of $\cA_Z ^{\dR}$ is made explicit in this section following \cite[\S 6]{BD17}, and the Frobenius structure is determined in \S~\ref{sec:compfrob}.
  
\subsection{Notation. }\label{subsec:notation}
Henceforth, $X$ is a smooth projective curve of genus $g>1$ over $\Q$, and $Y \subset X$ is an affine open subset. Let $b$ be a rational point of $Y$ which is integral at $p$. Suppose 
\[ 
  \# (X\backslash Y)(\overline{\Q})= d,
\] 
and let $K/\Q$ be a finite extension over which all the points of $D = X \backslash Y$ are defined. Choose a set $\omega_{0}, \ldots, \omega_{2g+d-2} \ \in \ \HH^0(Y,\Omega^1)$ of differentials on $Y$, satisfying the following properties:
\begin{itemize}
\item The differentials $\omega_0, \ldots, \omega_{2g-1}$ are of the second kind on $X$, and form a \textit{symplectic basis} of $\HH^1_{\dR}(X)$, in the sense that the cup product is the standard symplectic form with respect to this basis. 
\item The differentials $\omega_{2g}, \ldots, \omega_{2g+d-2}$ are of the third kind on $X$.
\end{itemize} 
We set $V_{\dR}(Y):= \HH^1 _{\dR}(Y)^*$, and let $T_0 ,\ldots , T_{2g+d-2}$ be the dual basis.

\subsection{The universal filtered connection $\cA_n^{\dR}$. }
\label{subsec:A2-filtration}
Let $K$ be a field of characteristic zero, and $\cC^{\dR}(X_{K})$ the category of unipotent vector bundles with connection on $X_{K}$. Our choice of base point $b\in X(K)$ makes $\cC ^{\dR}(X_K )$ into a unipotent Tannakian category, which has a fundamental group we denote by $\pi_1^{\dR}(X,b)$, see \S\ref{subsec:pathz}. We use the notation from the appendix and set 
\[
A_n^{\dR}(b) = A_n (\mathcal{C}^{\dR}(X_{\Q _p },b^* )),
\]
with associated path torsors $A_n^{\dR}(b,x)$. Let $\cA_n^{\dR}(b)$, or simply $\cA_n^{\dR}$, be the universal $n$-unipotent object, associated to the $\pi_1^{\dR}(X,b)$-representation $A_n^{\dR}(b)$ via the Tannaka equivalence, see \S~\ref{subsec:universal-objects}. This vector bundle carries a Hodge filtration, with the property that the $K$-vector space isomorphism
\[
x^* \cA _n ^{\dR}(b) \simeq A_n ^{\dR} (b,x), \qquad \forall x : \mathrm{Spec}(K) \lra X_{K}
\]
is an isomorphism of filtered vector spaces. For more details, see also \cite[pp. 98--100]{Kim09}.

\par To make it more explicit, we now describe a closely related bundle $\cA_n^{\dR}(Y)$ on an affine open $Y$, using the notation introduced in \S\ref{subsec:notation}. This bundle is of larger rank, but has the advantage of admitting a very simple description. Its relation with $\cA_n^{\dR}$ is given in Corollary \ref{cor:AX-and-AY}. To distinguish it more clearly from $\cA_n^{\dR}(Y)$, we will denote $\cA_n^{\dR}$ by $\cA_n^{\dR}(X)$ in this paragraph.

\par Set $A_n ^{\dR}(Y) :=\oplus _{i=0}^n V_{\dR} (Y)^{\otimes i}$ and define the connection
\begin{equation}\label{eqn:connection-AY}
\nabla _n \ :\ \cA_n ^{\dR}(Y) \ \lra \ \cA_n ^{\dR}(Y) \otimes \Omega ^1 _Y\ ,  \qquad \nabla_n (v \otimes 1) \ = \ \sum _{i=0}^{2g+d-2}-(T_i \otimes v)\otimes \omega _i,
\end{equation}
on $\cA_n ^{\dR}(Y):=A_n^{\dR}(Y) \otimes \mathcal{O}_Y$. Then $\cA_n^{\dR}(Y)$ is $n$-step unipotent, in the sense that it has a filtration
\[
 \bigoplus _{i=j}^n V_{\dR} (Y)^{\otimes i} \otimes \mathcal{O}_Y, \qquad \mbox{for} \ j = 0,1,\ldots, n
\]
by subbundles preserved by $\nabla$, where the graded pieces inherit the trivial connection. The following theorem, proved by Kim \cite{Kim09}, provides a universal property for the bundle $\cA_n^{\dR}(Y)$.

\begin{theorem}[Kim]\label{thm:kim_connection}
Let $\mathbf{1} = 1\otimes 1$ be the identity section of $\cA_n^{\dR}(Y)$. Then $(\mathcal{A}_n  ^{\dR} (Y),\mathbf{1})$ is a universal $n$-pointed object, in the sense of \S \ref{subsec:Tannakian}. That is, for any $n$-step unipotent vector bundle $\mathcal{V}$ with connection on $Y$, and any section $v$ of $\mathcal{V}$, there exists a unique map
\[ f : \cA_n^{\dR}(Y) \lra \mathcal{V} \qquad \mbox{such that} \ \ f(\mathbf{1}) = v.\] 
\end{theorem}

Although universal properties mean $(\mathcal{A}_n ^{\dR} (Y), \mathbf{1})$ is unique up to unique isomorphism, the trivialisation of the underlying vector bundle above is not unique, as it depends on a choice of elements of $\HH^0 (Y,\Omega ^1 )$ defining a basis of $\HH^1 _{\dR} (Y)$. The trivialisation has some relation to the algebraic structure of the spaces $A_n ^{\dR}(b,x)$, which we now explain. For $x \in Y(K)$, it gives is a canonical isomorphism
\[
s_0 (b,x) : A_n^{\dR}(Y) \ \lra \ A_n ^{\dR}(Y)(b,x) := x^* \cA_n^{\dR}(Y).
\]

\par The left hand side has a natural algebra structure, by viewing it as a quotient of the tensor algebra of $V _{\dR}(Y)$. On the other hand, for all $x_1 ,x_2 ,x_3 \in Y(K)$ we have (see Appendix \S \ref{subsec:Tannakian}) maps
\[
A_n ^{\dR}(Y)(x_2 ,x_3 )\times A_n ^{\dR}(Y)(x_1, x_2) \ \lra \  A_n ^{\dR}(Y)(x_1 ,x_3 ).
\]
coming from the surjections $\Hom (x_i ^* ,x_j ^* )\to A_n ^{\dR}(Y)(x_i ,x_j )$ and the composition of natural transformations
\[
\Hom (x_2 ,x_3 )\times \Hom (x_1 ,x_2 )\ \lra \ \Hom (x_1 ,x_3 ). 
\]
\begin{lemma}\label{lemma:fml}
For all $f_1 ,f_2 \in A_n^{\dR}(Y)$, and all $x_1 ,x_2 ,x_3 \in Y(K)$, 
\[
s_0 (x_1 ,x_3 )(f_2 f_1 )=s_0 (x_2 ,x_3 )(f_2 )s_0 (x_1 ,x_2 )(f_1 ).
\]
\begin{proof}
As in the appendix, let $\mathcal{C}^{\dR}(Y)_n $ denote the category of connections which are $i$-unipotent for $i\leq n$. Let $x_{1n},x_{2n},x_{3n}$ denote the fibre functors corresponding to the points $x_1,x_2 ,x_3$. Let $\alpha _{x_i ,x_j }$ denote the isomorphism $A_n ^{\dR}(Y)(x_i ,x_j )\simeq \Hom (x_{in},x_{jn})$. Then $\alpha _{x_2 ,x_3 }(s_0 (x_2 ,x_3 )(f_2 )) \in \Hom (x_{2n},x_{3n})$, and $s_0 (x_1 ,x_2 )(f_1 ) \in x_2 ^* \cA _n ^{\dR}(Y)$, and to prove the lemma it is enough to prove that
\[
(\alpha _{x_2 ,x_3 }(s_0 (x_2 ,x_3 )(f_2 )))(s_0 (x_1 ,x_2 )(f_1 ))=s_0 (x_1 ,x_3 )(f_2 f_1 )
\]
in $x_3 ^* \cA _n ^{\dR}(Y)$.
\par To prove this, note that there is a morphism of connections $F_1 : \cA _n ^{\dR}(Y) \to \cA _n ^{\dR}(Y)$ given by sending $v$ to $vf_1 $. Hence the lemma follows from commutativity of
\[
\begin{tikzpicture}
\matrix (m) [matrix of math nodes, row sep=3em,
column sep=3em, text height=1.5ex, text depth=0.25ex]
{x_2 ^* \cA _n ^{\dR}(Y) & x_3 ^* \cA _n ^{\dR}(Y)  \\
 x_2 ^* \cA _n ^{\dR}(Y) & x_3 ^* \cA _n ^{\dR}(Y). \\};
\path[->]
(m-1-1) edge[auto] node[auto]{$x_2 ^* F$} (m-2-1)
edge[auto] node[auto] { $\alpha$ } (m-1-2)
(m-2-1) edge[auto] node[auto] {$\alpha$ } (m-2-2)
(m-1-2) edge[auto] node[auto] {$x_3 ^* F$} (m-2-2);
\end{tikzpicture}
\]
where $\alpha :=\alpha _{x_2 ,x_3 }(s_0 (x_2 ,x_3 )(f_2 ))(\cA _n ^{\dR}(Y))$
\end{proof}
\end{lemma}

The following result describes the relation between $\cA_n^{\dR}(X)$ and $\cA_n^{\dR}(Y)$, see also \cite{BD17}. 

\begin{corollary}\label{cor:AX-and-AY}
The connection $\cA_n^{\dR}(X)|_Y$ is the maximal quotient of $\cA_n^{\dR}(Y)$ which extends to a holomorphic connection (i.e. without log singularities) on the whole of $X$.
\begin{proof}
It is enough to show that, for any surjection of left $\pi _1 ^{\dR}(Y,b)$-modules 
\[
p :A_n ^{\dR}(Y) \lra N,
\]
the associated connection $\mathcal{N}$ extends to a connection on $X$ without log singularities if and only if $p$ factors through the surjection $A_n ^{\dR}(Y) \lra A_n ^{\dR}(X)$. The latter occurs if and only if $N$ is the pullback of a left $\pi _1 ^{\dR}(X,b)$-module. The corollary follows by the Tannaka equivalence between left $\pi _1 ^{\dR}(X,b)$-modules, and unipotent connections on $X$.
\end{proof}
\end{corollary}

\subsection{The Hodge filtration on $\cA_n^{\dR}$. }
\label{subsec:computing-conn}
In what follows, we will need to explicitly compute the Hodge filtration of $\cA_2^{\dR}$, or rather of a certain quotient $\cA_Z$. To this end, we now state a characterisation of this Hodge filtration via a universal property, due to Hadian \cite{H11}.

\par Recall that a \textit{filtered connection} is defined to be a connection $(\mathcal{V},\nabla)$ on $X$, together with an exhaustive descending filtration 
\[
\ldots \supset \Fil^i \mathcal{V} \supset \Fil^{i+1}\mathcal{V} \supset \ldots
\]
satisfying the Griffiths transversality property
\[
\nabla (\Fil^i \mathcal{V})\subset (\Fil^{i-1}\mathcal{V})\otimes \Omega ^1 .
\]
A morphism of filtered connections is one that preserves the filtrations and commutes with $\nabla$. 

\par The universal $n$-unipotent bundle $\cA_n^{\dR}(b)$ is associated to the $\pi_1^{\dR}(X,b)$-representation $A_n^{\dR}(b)$, and there is a natural exact sequence of representations 
\[
0 \lra I^n / I^{n+1} \lra A_n^{\dR}(b) \lra A_{n-1}^{\dR}(b) \lra 0
\]
where the kernel $I^n / I^{n+1}$ has trivial $\pi_1^{\dR}(X,b)$-action. This means that the kernel
\[
\cA^{\dR}[n] \ := \ \mathrm{Ker}\left(\cA_n ^{\dR}(b) \lra \cA_{n-1}^{\dR}(b) \right) \ \simeq \ I^n /I^{n+1} \otimes \mathcal{O}_X ,
\]
is a trivial bundle with connection. The natural surjection $(I/I^2 )^{\otimes n} \lra I^n /I^{n+1}$ gives rise to a surjection $V_{\dR}^{\otimes n}\otimes \mathcal{O}_X \lra \cA^{\dR}[n]$. The Hodge filtration on $V_{\dR}$ gives $\cA^{\dR}[n]$ its structure of a filtered connection.  As explained in \cite{BD17}, the Hodge filtration on $\cA_n ^{\dR}(b)$ may now be characterised using Hadian's universal property, proved in \cite{H11}. 

\begin{theorem}[Hadian]\label{thm:Hadian}
For all $n>0$, the Hodge filtration $\Fil^{\bullet}$ on $\cA_n ^{\dR}(b)$ is the unique filtration such that 
\begin{itemize}
\item $\Fil^{\bullet}$ makes $(\cA_n ^{\dR}(b),\nabla )$ into a filtered connection, \\
\item The sequence
\[
V_{\dR}^{\otimes n}\otimes \mathcal{O}_X \lra \cA_n ^{\dR}(b) \lra \mathcal{A}_{n-1} ^{\dR}(b) \lra 0
\]
is a sequence of filtered connections. 
\item The identity element of $A_n^{\dR}(b)$ lies in $\Fil^0 A_n ^{\dR} (b)$.
\end{itemize}
\end{theorem}

\subsection{The filtered connection $\cA _Z $}
\label{subsec:bundle-AZ} 

\par As in~\S \ref{sec:heights}, a central role is played by a Tate class, which will come from an algebraic cycle on $X \times X$. Since the contribution to the $p$-adic height is entirely through its realisation as a $p$-adic de Rham class, we phrase things in this language. Henceforth, let 
\[
Z=\sum Z_{ij}\omega _i \otimes \omega _j \in \HH^1 _{\dR}(X)\otimes \HH^1 _{\dR} (X)
\]
be a nonzero cohomology class satisfying the following conditions:
\begin{enumerate}[(a)]
\item \label{pa} $Z$ is in $(\HH^1 _{\dR}(X)\otimes \HH^1 _{\dR}(X))^{\phi =p}$.
\item \label{pb} $Z$ is in $\Fil^{1}(\HH^1 _{\dR} (X) \otimes \HH^1 _{\dR} (X))$. 
\item \label{pc} $Z$ maps to zero under the cup product
\[
\cup \ : \ \HH^1 _{\dR} (X) \otimes \HH^1 _{\dR}(X) \lra \HH^2 _{\dR}(X).
\]
\item \label{pd} $Z$ maps to zero under the symmetrisation map
\[
\HH^1_{\dR}(X) \otimes \HH^1_{\dR} (X) \lra \Sym^2 \HH^1_{\dR} (X).
\]
\end{enumerate}
By property~\eqref{pd}, we may henceforth think of $Z$ as an element of $\HH^2_{\dR}
(\JacX_{\Q_p})$. It follows from Lemma~\ref{lemma:kunneth} 
that the Tate class associated to a \nice correspondence satisfies these properties. Though we will not need it in the sequel, the following result gives a converse to this statement.

\begin{lemma}\label{lemma:general_tate_class}
Let $Z$ be a class satisfying properties \eqref{pa}--\eqref{pd}. If $\rho (\JacX)=\rho (\JacX_{\Q_p})$, then there exists a \nice element of $\Pic(X\times X)\otimes \Q_p$ mapping to $Z$.
\begin{proof}
By the Tate conjecture for $\HH^2$ of abelian varieties over finite fields, property ~\eqref{pa} 
of $Z$ guarantees that it comes from a divisor on $\JacX_{\F_p }$. By the $p$-adic
Lefschetz (1,1)-theorem of Berthelot--Ogus \cite[\S 3.8]{BO83}, property~\eqref{pb}
implies that it lifts to something in $\NS (\JacX_{\Q _p })$. The hypotheses of the
theorem guarantee that this comes from a cycle on the Jacobian. Finally,
the element of $\Pic(X\times X)\otimes \Q_p$ corresponding to this cycle is \nice by property~\eqref{pc} of $Z$.
\end{proof}
\end{lemma}

\par We now come to the definition of the main object of this section and the next. Recall that we have a filtered $F$-isocrystal $\mathcal{A}_2^{\dR}$ on $X$, with an exact sequence
\begin{equation}\label{eqn:ses-filtration}
0 \lra \cA^{\dR}[2] \lra \mathcal{A}_2 ^{\dR} \lra \mathcal{A}_{1} ^{\dR} \lra 0,
\end{equation}
and there is an isomorphism 
\[
\cA^{\dR}[2] \simeq \mathrm{Coker}\left(\Q_p(1) \stackrel{\cup^*}{\lra} V_{\dR}^{\otimes 2}\right) \otimes \mathcal{O}_X. 
\]
Define $\cA_Z(b)$, or simply $\cA_Z$, to be the quotient of $\cA^{\dR}_2$ obtained by pushing out \eqref{eqn:ses-filtration} along
\[ 
Z \ :  \ V_{\dR} \otimes V_{\dR} \ \lra \ \Q_p(1),
\]
which by property~\eqref{pc} of $Z$ factors through $V_{\dR}^{\otimes 2} / \mathrm{Im}\ \cup^*$. The importance of this definition lies in the fact that non-abelian comparison isomorphisms naturally endow it with the structure of a filtered $F$-isocrystal, as we will see in \S\ref{sec:compfrob}, which has the property that for all Teichm\"uller points $x$ we have an isomorphism of filtered $\phi$-modules
\[ 
x^* \cA_Z \simeq \D_{\cris}(A_Z (b,x)).
\]
The $F$-structure on $\cA_Z$ is the subject of \S\ref{sec:compfrob}, and in the remainder of this section we will explicitly compute the connection and Hodge filtration on $\cA_Z$.

\par Using the results of \S \ref{subsec:A2-filtration}, we may describe the connection of $\cA_Z$ explicitly on the affine open $Y$. We use the notation of \S \ref{subsec:notation}, and denote the matrix of the correspondence $Z$ on $\HH^1_{\dR}(X)$ also by $Z$, where we act on column vectors. Then via Theorem \ref{thm:kim_connection}, we obtain a trivialisation
\[
s_0 :  \mathcal{O}_Y \otimes \left( \Q_p \oplus V_{\dR} \oplus \Q_p (1) \right) \stackrel{\sim}{\lra} \cA_Z |_Y. 
\]
By Corollary \ref{cor:AX-and-AY} and the explicit description of the connection on $\cA_n^{\dR}(Y)$ given in \eqref{eqn:connection-AY}, we have that the connection $\nabla$ on $\mathcal{A}_Z$ via the trivialisation $s_0$ is given by
\begin{equation}\label{eqn:connection-AZ-onY}
s_0^{-1} \circ \nabla \circ s_0 = d + \Lambda, \hspace{2cm} \mbox{where} \quad 
\Lambda := -\left( \begin{array}{ccc}
0 & 0 & 0 \\
\boldsymbol{\omega} & 0 & 0 \\
\eta & \boldsymbol{\omega}^{\intercal}Z & 0 \\
\end{array} \right),
\end{equation}
for some differential $\eta$ of the third kind on $X$. This differential is uniquely determined by the conditions that it is in the space spanned by $\omega_{2g}, \ldots , \omega_{2g+d-2}$, and that the connection $\nabla$ extends to a holomorphic connection on the whole of $X$, as in Corollary \ref{cor:AX-and-AY} (by a differential of the third kind, we mean a differential all of whose poles have order one).

\begin{Remark}In the notation above, and henceforth in this paper, block matrices are taken with respect to the $2$-step unipotent filtration with basis $\mathbf{1},T_0,\ldots,T_{2g-1},S$, and we use the notation $\boldsymbol{\omega}$ for the column vector with entries $\omega_0, \ldots, \omega_{2g-1}$. 
\end{Remark}

\subsection{The Hodge filtration of $\cA_Z$. }
\label{subsec:AZ-filtration}
We now make the Hodge filtration on $\cA_Z$ explicit. We will use Theorem \ref{thm:Hadian}, and our knowledge of the Hodge filtration on the first quotient $\cA_1$, to uniquely determine the Hodge filtration on $\cA_Z $ using the exact sequence of filtered connections
\[ 
0 \lra \Q_p (1) \otimes \mathcal{O}_X \lra \cA_Z \lra \cA_1 \lra 0. 
\] 
Since $\cA_1 \mid_Y \simeq \left(V_{\dR} \oplus \Q_p \right) \otimes \mathcal{O}_Y$ as filtered connections, we know that 
\[
\Fil^0 \cA_1 \mid_Y = \left(\Q_p \oplus \Fil^0 V_{\dR} \right) \otimes \mathcal{O}_Y, 
\]
Recall the explicit description of the connection $\nabla$ on the restriction of $\cA_Z$ to $Y$, given by \eqref{eqn:connection-AZ-onY} with respect to the basis $\mathbf{1},T_0,\ldots,T_{2g-1},S$. In analogy with the notation of~\S \ref{subsec:local_height_p}, we may specify the Hodge filtration by giving an isomorphism of filtered vector bundles
\[
s^{\Fil} \ :\  \left( \Q_p \oplus V_{\dR} \oplus \Q_p (1) \right) \otimes \mathcal{O}_Y \ \stackrel{\sim}{\lra} \ \cA_Z |_Y,
\]
where the filtration on the left hand side is induced from the Hodge filtration on its graded pieces. Such a morphism $s^{\Fil}$ is uniquely determined by the vector $\boldsymbol{\beta}_{\Fil}$ and $\gamma_{\Fil} \in \HH^0(Y,\mathcal{O}_Y)$ in
\[
s_0^{-1} \circ s^{\Fil} = \left( \begin{array}{ccc}
1 & 0 & 0 \\
0 & 1 & 0 \\
\gamma_{\Fil} & \boldsymbol{\beta}_{\Fil}^{\intercal} & 1 \\
\end{array}
\right).
\]

The conditions imposed by Theorem \ref{thm:Hadian} determine $\gamma_{\Fil}$ and $\boldsymbol{\beta}_{\Fil}$ uniquely, as we will now explain. At each point $x$ in $D = X-Y$, defined over $K$, let 
\[
s_x \ : \ 
\Big(
(\Q_p \oplus V_{\dR}\oplus \Q _p (1))\otimes K \llbracket t_x \rrbracket,\ d \ \Big)
\stackrel{\sim}{\lra} 
\left( \cA_Z |_{K \llbracket t_x \rrbracket },\ \nabla \ \right) 
\]
be a trivialisation of $\cA_Z$ in a formal neighbourhood of $x$, with local parameter $t_x$. The difference between the bundle trivialisations defines a gauge transformation 
\[
G_x := s_x^{-1}\circ s_0 \ \in \  \mathrm{End}\big( \Q _p \oplus V_{\dR} \oplus \Q _p (1) \big) \otimes K (\! (t_x )\! )
\]
satisfying
\begin{equation}\label{eqn:gauge_transformation}
G_x^{-1}dG_x = \Lambda.
\end{equation}
Conversely, any such $G_x$ defines a trivialisation $s_x$. Expanding out \eqref{eqn:gauge_transformation} shows that $G_x$ is of the form
\begin{equation}\label{eqn:gauge_infty}
G_x = \left( \begin{array}{ccc}
1 & 0 & 0 \\
\boldsymbol{\Omega}_x & 1 & 0\\
g_x & \boldsymbol{\Omega}^{\intercal}_xZ & 1 \\
\end{array} \right),\quad
\mbox{ where }
\left\{ 
\begin{array}{lll}
d\boldsymbol{\Omega}_x &=& \boldsymbol{\omega}\\
dg_x &=& d\boldsymbol{\Omega}_x^{\intercal}Z\boldsymbol{\Omega}_x -
\boldsymbol{\Omega}_x^{\intercal}(Z-Z^{\intercal})\boldsymbol{\omega} -  \eta . \\
\end{array} 
\right.
\vspace{.2cm}
\end{equation}
Equivalently, the gauge transformation $G_x$ defines a basis of formal horizontal sections of $\cA_Z$ at $x$. By Theorem \ref{thm:Hadian}, $\Fil^0 \cA_Z \mid_Y$ extends to a bundle on $X$, which results in the condition that the functions in $\boldsymbol{\beta}_{\Fil}$ extend to holomorphic functions on $X$, and are hence constant, as well as the condition 
\[
g_x - \gamma_{\Fil} - \boldsymbol{\beta}^{\intercal}_{\Fil} \cdot \boldsymbol{\Omega}_x   \quad \in \quad K [\! [t_x ]\! ],
\]
see also \cite[\S 6.4]{BD17}. The existence and uniqueness follows from the following lemma. We omit its proof, which is an elementary argument using Riemann--Roch, and refer the reader to \cite[Lemma 25]{BD17} for a similar argument.
\begin{lemma}\label{lemma:silly_RR}
Given any tuple $(g_x ) \in \prod _{x \in D} K (\! (t_x )\! )$, there exists a unique vector of constants $\boldsymbol{\beta} \in K^{2g}$ and a function $\gamma \in \HH^0 (Y,\mathcal{O})$, unique modulo constants, such that for all $x\in D$,
\[
g_x - \gamma - \boldsymbol{\beta}^{\intercal} \cdot \boldsymbol{\Omega}_x   \quad \in \quad K [\! [t_x ]\! ].
\] 
\end{lemma}

By the above lemma, we can determine the vector of constants $\boldsymbol{\beta}_{\Fil}$ uniquely, and $\gamma_{\Fil}$ is uniquely determined by the additional condition that $\gamma_{\Fil}(b) = 0$. In summary, this gives the following algorithm for computing the Hodge filtration on $\mathcal{A}_Z ^{\dR}$.
\begin{enumerate}[(i)]
\item Compute the differential $\eta$ as in~\eqref{eqn:connection-AZ-onY}, as the unique linear combination of $\omega_{2g}, \ldots , \omega_{2g+d-2}$ such that 
\[
d\boldsymbol{\Omega}_x^{\intercal}Z\boldsymbol{\Omega}_x -
\boldsymbol{\Omega}_x^{\intercal}(Z-Z^{\intercal})\boldsymbol{\omega} -  \eta
\]
has vanishing residue at all $x \in X \backslash Y$.\\
\item For all $x \in X \backslash Y$, compute power series for $\boldsymbol{\omega}_x$ and $\eta$ up to large enough precision, which means at least $\pmod{t_x^{d_x}}$, where $d_x$ is the order of the largest pole occurring. Use this to solve the system of equations \eqref{eqn:gauge_infty} for $g_x$ in $K (\! (t_x )\! )/K [\! [t_x ]\! ]$. \\
\item Compute the constants $\boldsymbol{\beta}_{\Fil}$ and function $\gamma_{\Fil}$ characterised $\gamma_{\Fil}(b)=0$ and 
\[
g_x - \gamma_{\Fil} - \boldsymbol{\beta}^{\intercal}_{\Fil} \cdot \boldsymbol{\Omega}_x   \quad \in \quad K [\! [t_x ]\! ], \qquad \forall x \in X\backslash Y.
\]
\end{enumerate}

\section{Explicit computation of the $p$-adic height II: Frobenius}\label{sec:compfrob}

The preceding section gives a computationally feasible method for computing the Hodge filtration on $A_Z ^{\dR} (b,x)$. To complete the computation of the filtered $\phi $-module structure on $A_Z ^{\dR} (b,x)$ we need to describe the Frobenius structure. When $X$ is hyperelliptic, such a description is given in \cite[\S 6]{BD17} in terms of iterated Coleman integrals. It is also explained that the iterated integral description will not determine the $\phi $-action in general, but that for a general curve an alternative approach would be to compute the Frobenius structure on the isocrystal $\cA_Z^{\rig }$ (relative to the basepoint $b$), and pull this back at $x$. In this section, we carry out this strategy.

\subsection{The Frobenius structure on $\cA_n^{\rig}$ and iterated Coleman integration }
\label{subsec:filteredF}

We now describe how $\cA_n^{\rig}$ obtains the structure of an $F$-isocrystal. Some background on unipotent isocrystals can be found in \S \ref{subsec:background_isocrystals}. Let $\mathcal{U} \subset \cY_{\F_p}$ be a Zariski open subset, and let $\mathfrak{X}, \mathfrak{Y}$ denote the formal completions of $\cX$ and $\cY$ along their special fibres. Choose a lift of Frobenius
\[
\phi:\mathcal{W} \ \lra \ \mathfrak{X}
\]
defined on a strict open neighbourhood $\mathcal{W}$ of $]\mathcal{U}[$.

\begin{Definition}
A \textit{Frobenius structure} on an overconvergent isocrystal $\mathcal{V}$ on $Y \subset \cX_{\F_p}$ is an isomorphism
\[
\Phi : \phi^* \mathcal{V}\stackrel{\sim}{ \lra }\mathcal{V},
\]
of overconvergent isocrystals on $Y$, where $\phi : Y \to Y$ is absolute Frobenius. The pair $(\mathcal{V}, \Phi)$ is referred to as an overconvergent $F$-isocrystal. 
\end{Definition}

\par The main example of importance for us is the $F$-isocrystal structure on $\cA_n^{\rig }(\overline{b})$. Using the notation of the appendix, we set $\cC^{\rig}(\cX_{\F_p})$ to be the category of unipotent isocrystals on $\cX_{\F_p}$, with fundamental group $\pi_1^{\rig}(\cX_{\F_p},\overline{b})$ and universal $n$-step unipotent objects $\cA_n^{\rig}(\overline{b})$. As explained in \S \ref{subsec:nonabeliancomparison}, the non-abelian Berthelot--Ogus equivalence induces an isomorphism of rigid analytic connections
\[
\iota \ : \ \cA_n^{\dR} (b)^{\an }\stackrel{\sim }{\lra} \cA_n^{\rig }(\overline{b}),
\]
and similarly for $\cA _n ^{\dR}(Y)$.
The Frobenius action on $\cC^{\rig }(\cX_{\F_p})$ induces a Frobenius structure $\Phi$ on $\cA_n^{\rig }(\overline{b})$, which induces the Frobenius action on $\overline{z}^* \cA _n ^{\rig}(\overline{b})$ via pullback. 

\begin{lemma}\label{lemma:compromise}
The Frobenius structure on $\cA _n ^{\rig }(\overline{b})$ is the unique morphism 
\[
\Phi _n :\phi ^* (\cA _n ^{\rig }(\overline{b})) \ \lra \ \cA _n ^{\rig }(\overline{b})
\]
which, in the fibre at $\overline{b}$, sends 1 to 1.
\begin{proof}
Unicity is clear, since a morphism of $n$-unipotent universal objects is uniquely determined by where it sends $1\in \overline{b}^* \cA _n ^{\rig }(\overline{b})$ (see \S\ref{subsec:Tannakian}). The fact that the Frobenius structure has this property follows from Lemma \ref{lemma:too_far}, since the Frobenius endomorphism in $\Hom (\overline{b}^*,\overline{b}^*)$ is a morphism of unital algebras.
\end{proof}
\end{lemma}

\par By Lemma \ref{lemma:too_far}, working with the rigid triple $(Y_{\F _p },X_{\F _p
},\mathfrak{X})$ and Frobenius lift $\phi :\mathcal{W}\to X_{\Q _p }$, the Frobenius
action on $A_n ^{\rig }(\overline{b},\overline{z})$ may be identified with the
endomorphism $z_0 ^* \Phi _n $ acting on $z_0 ^* \cA _n ^{\rig }(b_0 )$, where $z_0 $ and
$b_0 $ are Teichm\"uller lifts of $\overline{b}$ and $\overline{z}$ and $\Phi _n $ is the Frobenius structure defined in Lemma \ref{lemma:too_far}. Hence in this case thecomputation of the Frobenius action on $A_n ^{\dR}(Y)(b_0 ,z_0 )$ is reduced to the computation of the Frobenius structure. For general $b$ and $z$, one may straightforwardly reduce the problem of determining the $\phi $ action on $A_n ^{\dR}(b,z)$ to the computation of the Frobenius structure on $\cA _n ^{\rig }$ via parallel transport on residue disks, which we now explain. As $\cA _n ^{\dR}$ is calculated as a quotient of $\cA _n ^{\dR }(Y)$, it is enough to describe how to carry out this reduction for $A_n ^{\dR}(Y,b,z)$.
Let $G_n (b ,z) \in \End _{\Q _p } (\oplus _{i=0}^n V_{\dR}(Y)^{\otimes i})$ be the matrix defined by 
\[
\Phi _n (b,z)= s_0 (b_0 ,z_0 )\circ G_n (b,z) \circ s_0 (b_0 ,z_0 )^{-1}.
\]
For $v \in \oplus _{i=0}^n V_{\dR}(Y)^{\otimes i}$, define $E(v_1 ,v_2 )\in \End _{\Q _p } (\oplus _{i=0}^n V_{\dR}(Y)^{\otimes i})$ to be the endomorphism
$
v\mapsto v_1 vv_2 .
$
If $z_1 ,z_2$ are points of $]Y[$ which are congruent modulo $p$, then we have a canonical parallel transport isomorphism
$
T_{z_1 ,z_2 }:z_1 ^* \stackrel{\simeq }{\longrightarrow }z_2 ^*$  (see Appendix \ref{sec:Tannakian}).

Define $C_n (z_1 ,z_2 )\in \oplus _{i=0}^n V_{\dR}(Y)^{\otimes i}$ by 
\[ 
T_{z_1 ,z_2 }(\cA _n^{\dR}(Y))(s_0 (z_1 ))(w)=s_0 (z_2 )(C_n (z_1 ,z_2 ).w).
\]  Hence 
\[
C_n (z_1,z_2 )=1+\sum _w \int ^z _b w(\omega _0 ,\ldots ,\omega _{2g+2d-2})w ,
\]
where the sum is over all words $w$ in $\{T_0 ,\ldots ,T_{2g+d-2} \}$ of length at most $n$, and where $w(\omega _0 ,\ldots ,\omega _{2g+d-2})$ is defined to be the word in $\{ \omega _0 ,\ldots ,\omega _{2g+d-2} \}$ obtained by substituting $\omega _i $ for $T_i $. 

\begin{lemma}\label{total_complete_utter_shitemare}
For all $b,z \in Y(\Z _p )$, with Teichm\"uller representatives $b_0$ and $z_0$, and $w\in \oplus _{i=0}^n V_{\dR}(Y)^{\otimes i}$,
\[
G_n (b,z)(s_0 (b,z)(w))=s_0 (b,z)(EG_n (b_0 ,z_0 ) E^{-1}),
\]
where $E:=E(C_n (z_0 ,z),C_n (b,b_0 ))$.
\begin{proof}
 By non-abelian Berthelot-Ogus, we obtain a Frobenius action on $A_n ^{\dR}(b,z)$, which explicitly is the composite
\[
T_{z_0 ,z}\circ z_0 ^* \cA _n ^{\rig }(Y) \circ T_{b,b_0 }.
\]
Hence the result follows from Lemma \ref{lemma:fml}, since composition of functors
\[
\Hom (z_0 ^* ,z ^* )\times \Hom (b_0 ^* ,z_0 ^* )\times \Hom (b^* ,b_0 ^* )\to \Hom (b^* ,z^* )
\]
acting on $\cA _n ^{\dR}(Y)$
corresponds via the maps $s_0 $ to multiplication in $\oplus _{i=0}^n V_{\dR}(Y)^{\otimes i}$.
\end{proof}
\end{lemma}

\par More generally, for any $z_1 ,z_2 $, as explained in \cite[3.2]{bES02}, $\pi _1 ^{\dR}(Y_{\Q _p };z_1 ,z_2 )$ contains a unique Frobenius invariant element, and we denote its image in $A_n ^{\dR}(Y_{\Q _p };z_1 ,z_2 )$ by $s_0 (z_1 ,z_2 )(C_n (z_1 ,z_2 ))$. By definition of Besser's iterated integration on curves, we have
\[
C_n (z_1 ,z_2 )=1+\sum _w \int ^{z_2 } _{z_1 } w(\omega _0 ,\ldots ,\omega _{2g+2d-2})w.
\]
For any $z_1 ,z_2 ,z_3 ,z_4$, we obtain a Frobenius-equivariant isomorphism
\begin{align}\label{eq:helpful}
A_n ^{\dR}(Y)(z_1 ,z_2 ) & \to A_n ^{\dR}(Y)(z_3 ,z_4 ) \\
s_0 (z_1,z_2 )(v) & \mapsto s_0 (z_3 ,z_4 )(C_n (z_2 ,z_4 )vC_n (z_3 ,z_1 ))
\end{align}
for $v$ in $\oplus _{i=0}^n V_{\dR}(Y)^{\otimes i}$.

Hence given the Frobenius structure on $\cA _n ^{\rig }$, one may compute iterated Coleman
integrals of depth $n$ between any 2 $\Q _p $-points of $Y$. As explained in \cite[\S
6.6]{BD17}, when $X$ is hyperelliptic, one can determine the Frobenius action on $A _2
^{\dR}(b,b)$ when $b$ is a Weierstrass point, and using equation \eqref{eq:helpful}
$\boldsymbol{\beta}_\phi$ may be related to Coleman integrals from $b$ to $w(b)$, where $w$ denotes the hyperelliptic involution. 
However in general the Frobenius structure will contain more information.

\subsection{The Frobenius structure on $\cA_Z^{\dR}$. }
\label{subsec:AZ-Frobenius}

We will now describe the structure of $F$-isocrystal on the analytification of $\cA_Z^{\dR}(b)$, by making explicit the Frobenius structure on $\cA_Z^{\rig}(\overline{b})$, as well as the isomorphism 
\[
\iota\ : \ \cA_Z^{\dR}(b)^{\an} \stackrel{\sim}{\lra} \cA_Z^{\rig}(\overline{b})
\]
provided by the Berthelot--Ogus comparison isomorphism, see \S \ref{subsec:nonabeliancomparison}. Recall that in our computation of the Hodge filtration, we made a choice of trivialisation of vector bundles on $Y \subset X$
\[
s_0 :  \mathcal{O}_Y \otimes \left( \Q_p \oplus V_{\dR} \oplus \Q_p (1) \right) \stackrel{\sim}{\lra} \cA_Z |_Y. 
\]

\par The choice of a \nice correspondence $Z$ defines, as we have shown in the \'etale and de Rham realisations, a quotient $\cA_Z^{\rig}(\overline{b})$ of the universal $2$-step unipotent isocrystal $\cA_2^{\rig}(\overline{b})$. To describe the Frobenius structure on $\cA_Z^{\rig}(\overline{b})$, we compute locally with respect to a lift of Frobenius. The connections on $\cA_Z^{\rig}(\overline{b}) \mid_Y$ and $\phi^*\cA_Z^{\rig}(\overline{b}) \mid_Y$ are described with respect to the trivialisation $s_0$ by equation \eqref{eqn:connection-AZ-onY}, and are by definition equal to $d+\Lambda$ and $d+\Lambda_{\phi}$, where
\[
\Lambda_{\phi} = 
-\left( \begin{array}{ccc}
0 & 0 & 0 \\
\phi^*\boldsymbol{\omega} & 0 & 0 \\
\phi^*\eta & \phi^*\boldsymbol{\omega}^{\intercal}Z & 0 \\
\end{array} \right).\]
Henceforth, we set $\phi^*\boldsymbol{\omega} = F\boldsymbol{\omega} + d\mathbf{f}$ for a column vector $\mathbf{f}$ with entries in $\HH^0(]Y[,j^\dagger \mathcal{O}_Y)$, uniquely determined by the condition that $\mathbf{f}(b) = \mathbf{0}$. To make the $F$-structure explicit, we need to find an invertible $d \times d$-matrix $G$ with entries in $\HH^0(]Y[,j^\dagger \mathcal{O}_Y)$, such that
\[
\Lambda_{\phi}G + dG = G\Lambda.
\]
It is a straightforward calculation using the relation $F^{\intercal}ZF=pZ$ to check that the matrix\vspace{.2cm}
\begin{equation}\label{eqn:F-structure}
G = \left( \begin{array}{ccc}
1 & 0 & 0 \\
\mathbf{f} & F & 0\\
h & \mathbf{g}^\intercal & p \\
\end{array} \right), \mbox{ where }
\left\{ 
\begin{array}{lll}
d\mathbf{g}^\intercal & = & d\mathbf{f}^\intercal ZF, \\
dh & = & \boldsymbol{\omega}^\intercal F^\intercal Z\mathbf{f} + d\mathbf{f}^\intercal Z\mathbf{f}-\mathbf{g}^\intercal \boldsymbol{\omega}
+\phi^*\eta - p\eta, \\
h(b) & = & 0, \\
\end{array} 
\right.
\vspace{.2cm}
\end{equation}
induces the required gauge transformation. The algorithms of Tuitman \cite{Tui16, Tui17} may be used to explicitly solve the above equations and uniquely determine $\mathbf{f}, \mathbf{g}$ and $h$, see Section \ref{sec:Xs13}. 

\par Finding the Frobenius structure on $\cA_Z^{\rig}(\overline{b})$ amounts to straightforward linear algebra now that we have made explicit the action of Frobenius on both sides. Let $x_0$ be the unique Teichm\"uller point in $]x[$. Then the unique Frobenius-equivariant isomorphism 
\[
s^{\phi}(b,x) : \Q_p \oplus V_{\dR}\oplus \Q_p(1) \stackrel{\sim}{\lra} x^*\cA_Z^{\dR}(b)
\]
is determined for $x = x_0$ by the matrix
\begin{equation}\label{eqn:frob-triv-Teich}
s_0 ^{-1}(x_0 )\circ s^{\phi}(b,x_0) := \left( \begin{array}{ccc}
1 & 0 & 0 \\
(I-F)^{-1}\mathbf{f} & 1 & 0\\
\frac{1}{1-p}\left(\mathbf{g}^{\intercal}(I-F)^{-1}\mathbf{f} +h\right) & \mathbf{g}^{\intercal}(F-p)^{-1} & 1 \\
\end{array} \right),
\end{equation}
and for arbitrary $x$ in $]x_0 [$ by the matrix
\begin{equation}\label{eqn:frob-triv-non-Teich}
s_0 ^{-1}(x) \circ s^\phi (b,x) = I(x,x_0) \circ s_0 ^{-1}(x_0 )\circ s^\phi (b,x_0 ).
\end{equation}

\subsection{Back to $p$-adic heights. }
\label{subsec:assoc-FFI}

Recall that for the intended Diophantine application, we set out to compute the function 
\[
\theta \ :\ X(\Q _p ) \ \lra \ \Q _p \ ; \ x \ \longmapsto \ h_p \left( A_Z (b,x)\right)
\]
in order to obtain an explicit finite set of points in $X(\Q_p)$ containing $X(\Q)$. In \S \ref{sec:heights}, we reduced this question to finding an explicit description of the filtered $\phi$-module $\D_{\cris }(A_Z (b,x))$. Since we have a commutative diagram with exact rows
\begin{equation}\label{eq:comparison_diagram}
\begin{tikzpicture}[baseline=(current  bounding  box.center)]
\matrix (m) [matrix of math nodes, row sep=2.5em,
column sep=1.5em, text height=1.5ex, text depth=0.25ex]
{0 & \D_{\cris } \left( \Coker \left(\Q _p (1) \stackrel{\cup ^* }{\longrightarrow }V^{\otimes 2} \right) \right) & \D_{\cris }(A_2^{\et} (b,x)) & \mathbf{D}_{\cris }(A_1^{\et} (b,x)) & 0  \\
0 & \Coker \left( \Q_p(1)\stackrel{\cup ^* }{\longrightarrow }V_{\dR} ^{\otimes 2} \right) & A_2 ^{\dR }(b,x) & A_1 ^{\dR} (b,x) & 0 \\};
\path[->]
(m-1-1) edge[auto] node[auto]{} (m-1-2)
(m-2-1) edge[auto] node[auto]{} (m-2-2)
(m-1-2) edge[auto] node[auto]{$\simeq $} (m-2-2)
edge[auto] node[auto] { } (m-1-3)
(m-1-3) edge[auto] node[auto] {$\simeq $} (m-2-3)
edge[auto] node[auto] { } (m-1-4)
(m-1-4) edge[auto] node[auto] {$\simeq $} (m-2-4)
edge[auto] node[auto] { } (m-1-5)
(m-2-2) edge[auto] node[auto]{} (m-2-3)
(m-2-3) edge[auto] node[auto]{} (m-2-4)
(m-2-4) edge[auto] node[auto]{} (m-2-5)
;
\end{tikzpicture}
\end{equation}
it follows that $\D_{\cris }(A_Z(b,x))$ may be identified with the filtered $\phi$-module $A_Z^{\dR} (b,x)$ obtained by pushing out the bottom exact sequence of diagram \eqref{eq:comparison_diagram} by the map 
\[
\cl _Z ^* :V_{\dR}\otimes V_{\dR} \lra \Q_p(1),
\]
where we implicitly use the fact that the $p$-adic comparison isomorphism is compatible with cycle class maps, and the fact that $\cl_Z^*$ factors through $\Coker (\cup ^*)$. It follows that $\D_{\cris }(A_Z (b,x)) \simeq x^* \cA_Z^{\dR}(b)$, as filtered $\phi$-modules. Since we have computed the filtered $F$-isocrystal structure on $\cA_Z^{\dR}(b)$ explicitly, we are now in a position to put everything together and obtain an explicit description for the $p$-adic height function $\theta$, which can be computed in practice.

\par Recall that in \S \ref{sec:comphodge}, we obtained a simple algorithm for determining the Hodge filtration on $\cA_Z(b)$ via Hadian's universal property. The steps are outlined at the end of \ref{subsec:AZ-filtration}, and yield a matrix
\[
s_0^{-1} \circ s^{\Fil}(x) = \left( \begin{array}{ccc}
1 & 0 & 0 \\
0 & 1 & 0 \\
\gamma_{\Fil}(x) & \boldsymbol{\beta}_{\Fil}^\intercal & 1 \\
\end{array}
\right).
\]
The Frobenius structure is computed explicitly on the Teichm\"uller representative $x_0$ of $x$ using the algorithms of Tuitman \cite{Tui16, Tui17} to solve the system of equation \eqref{eqn:F-structure}, and then for $x$ via explicit integration in the residue disk, to yield a matrix of locally analytic functions on $]\mathcal{U}[$
\[ 
s_0^{-1} \circ s^{\phi}(x) := \left( \begin{array}{ccc}
1 & 0 & 0 \\
\boldsymbol{\alpha}_{\phi}(x) & 1 & 0\\
\gamma_{\phi} (x) & \boldsymbol{\beta}^\intercal_{\phi}(x) & 1 \\
\end{array} \right)
\]
as described by \eqref{eqn:frob-triv-Teich} and \eqref{eqn:frob-triv-non-Teich}. As an immediate consequence of equation \eqref{eqn:height-decomposition}, we obtain
\begin{lemma}\label{lemma:ht_expl}
For any $x \in X(\Q_p) \cap \ ]\mathcal{U}[$, the local $p$-adic height of $A_Z(b,x)$ is given by
\[
h_p (A_Z(b,x))=
\chi_p \left( \gamma_{\phi}(x) - \gamma_{\Fil}(x) -\boldsymbol{\beta}^\intercal_{\phi} \cdot s_1 (\boldsymbol{\alpha}_{\phi}) (x) - \boldsymbol{\beta }^\intercal_{\Fil} \cdot s_2 (\boldsymbol{\alpha}_{\phi})(x) \right).
\]
\end{lemma}
In conclusion, given any particular example of a curve $X$ that satisfies the running assumptions, to compute a finite set of $p$-adic points containing $X(\Q)$ it is enough to choose a \nice correspondence $Z$ and to do the computation outlined above. This is carried out in \S \ref{sec:Xs13} for the curve $X = X_{\sC}(13)$.

\subsection{A trick for dealing with leftover residue disks}\label{trick}

The process described above gives a way of computing a finite set containing $X(\Q  )\cap ]\mathcal{U}[$. The remaining disks contain points where the basis differentials have poles, or where the Frobenius lift is not defined. One approach to dealing with the remaining points is to choose a different bases of de Rham cohomology and different Frobenius lifts until the whole of $X(\F _p )$ is covered. In this subsection we describe an alternative approach which allows one to deal with a leftover residue disk which contains a rational point $b'$, and for which the differentials have no poles. The starting point is the observation that the set $X(\Q _p )_{\U}$ is indendendent of $b'$. The computation of the Hodge filtration with respect to the basepoint $b'$ may be carried out straightforwardly, so the nontrivial issue is the computation of the Frobenius structure on $\cA ^{\rig }_Z $ with respect to this new basepoint. Since we are only interested in the pullback of $G$ at the Teichm\"uller representatives of the remaining disks, we may take those to be basepoints.

\begin{lemma}\label{lemma:total_complete_utter_nightmare}
Let $b' \in X(\Q_p)$ lie on a residue disk where none of the basis differentials have poles. Then 
\[
s_0^{-1}(b') \circ s^\phi (b',b')=\left( \begin{array}{ccc} 1 & 0 & 0 \\
0 & 1 & 0 \\
0 & \boldsymbol{\beta }_\phi^\intercal(b) + 2\int ^{b' }_b \boldsymbol{\omega }^\intercal Z & 1 \\
\end{array}
\right).
\]
\begin{proof}
This follows from equation \eqref{eq:helpful}, via the quotient map
\[
A_2 ^{\dR}(Y)(b')\to A_Z ^{\dR}(b').
\]
To apply equation \eqref{eq:helpful}, we first have to describe the algebra structure of $A_Z ^{\dR}(b,b)$ thought of as a quotient of $A_2 ^{\dR}(b,b)$. Given an element $(a,\mathbf{b},c)$ of $A_Z ^{\dR}(b,b) \simeq \Q _p \oplus V_{\dR} \oplus \Q _p (1)$, its left and right actions on $A_Z ^{\dR} (b,b)$ are given by the matrices 
$\left( \begin{array}{ccc}
a & 0 & 0 \\
\mathbf{b} & a & 0 \\
c & \mathbf{b}^\intercal Z & a \\
\end{array} \right) $
and $\left( \begin{array}{ccc}
a & 0 & 0 \\
\mathbf{b} & a & 0 \\
c& -\mathbf{b}^\intercal Z & a \\
\end{array} \right)$
respectively.
Hence from equation \eqref{eq:helpful}, we obtain that $s_0^{-1}(b') \circ s^\phi (b',b')$
is equal to
\[
 \left( \begin{array}{ccc} 1 & 0 & 0 \\
\int ^{b'}_b \boldsymbol{\omega } & 1 & 0 \\
\int ^{b'}_b \eta +\int ^{b'}_b \boldsymbol{\omega }^\intercal  Z \boldsymbol{\omega } & \int ^{b' }_b \boldsymbol{\omega }^\intercal Z & 1 \\
\end{array}
\right)
\left( \begin{array}{ccc}
1 & 0 & 0 \\
0 & 1 & 0 \\
0 & \boldsymbol{\beta }_\phi^\intercal(b) & 1
\end{array} \right)
\left( \begin{array}{ccc} 1 & 0 & 0 \\
\int ^{b'}_b \boldsymbol{\omega } & 1 & 0 \\
\int ^{b'}_b \eta +\int ^{b'}_b \boldsymbol{\omega }^\intercal  Z \boldsymbol{\omega } &  -\int ^{b' }_b \boldsymbol{\omega }^\intercal Z & 1 \\
\end{array}
\right) ^{-1}.
\]
\end{proof}
\end{lemma}


\section{Example: $X_{s} (13)$}
\label{sec:Xs13}
As in the introduction, we denote by $X_{\sC}(\ell)$ the modular curve associated to the
normaliser of a split Cartan subgroup of $\GL_2(\F_{\ell})$, where $\ell$ is a prime
number. 
This curve can be defined over $\Q$ and 
there is a $\Q$-isomorphism $X^+_0(\ell^2) \simeq X_{\sC}(\ell)$ coming from conjugation 
  with $\left(\begin{smallmatrix}
\ell & 0 \\ 0 & 1 \end{smallmatrix}\right)$, where  $X^+_0(\ell^2)$ is the quotient of
$X_0(\ell^2)$ by the Atkin-Lehner involution $w_{\ell^2}$.
  
  \par
In this section, we compute the rational points on $X=X_{\sC}(13)$, which is a curve of
genus~3.
We show in~\S\ref{Rank} that we have $\rho(\JacX_{\Q}) = \rk(\JacX/\Q)= 3$ and
in~\S\ref{sec:reduction} that $X$ has potentially good reduction everywhere.
Hence Corollary~\ref{cor:qc_pair_pg} implies that 
\[
\Upsilon_Z = (\theta_Z, \{0\})
\]
is a quadratic 
Chabauty pair for $X$, where $Z \in \Pic(X \times X)
\otimes \Q_p$ is \nice and $\theta_Z(x) = h_p (A_Z (b,x))$. We work at the prime $p=17$ and we use $2 = g+\rho-1-r$ independent \nice $Z$; by
Lemma~\ref{lemma:qc_pair} we get two finite sets of 17-adic points which contain $X(\Q)$. Their intersection turns out to be exactly $X(\Q)$, which proves Theorem~\ref{thm:main_thm}.

\begin{Remark}
The choice of the prime 17 is somewhat arbitrary. For primes less than 11, our chosen basis
of de Rham cohomology is not $p$-integral, and at $p=13$ the curve has bad reduction. 
\end{Remark}

\begin{Remark}
Because $X$ is not hyperelliptic, the techniques from~\cite{BD17} for computing $\theta(x)$
are not applicable, so below we use the results of Sections~\ref{sec:comphodge} and~\ref{sec:compfrob}.
\end{Remark}
\subsection{Ranks. }\label{Rank}
Modular symbol routines as implemented in {\tt Magma}~\cite{BCP97}
allow us to compute the space of weight~2 cuspforms for $\Gamma_0^+(169)$, and it turns
out that their eigenforms form
a single Galois conjugacy class defined over $\Q(\zeta_7)^+$.
An explicit eigenbasis is given by the Galois conjugates of the form $f$, with $q$-expansion
\[ f = q+\alpha q^2 +(-\alpha^2 -2\alpha )q^3 +(\alpha ^2-2)q^4 +(\alpha ^2 +2\alpha -2)q^5 +(-\alpha -1)q^6 + \ldots, \]
where $\alpha$ is a root of $x^3+2x^2-x-1$. We conclude by \cite[Theorem 7.14]{Shi70} and \cite[Corollary 4.2]{Rib80} that $\End( \JacX) \otimes \Q \simeq K := \Q(\zeta_7)^+$, and $\JacX$ is a simple abelian threefold.
Therefore $\rho(\JacX_{\Q}) =3$. 

\begin{prop}\label{P:Rank}
We have $\mathrm{rk}(\JacX/\Q) = 3$.
\begin{proof}
Let $A_f$ denote the modular abelian variety $A_f=
J_0(169)/I_f$ associated to $f$, where $I_f$ is the annihilator of $f$ in the Hecke
algebra $\mathbb{T}$ acting on $J_0(169)$. Then $A_f$ is an optimal quotient of $J_0(169)$
in the sense that the kernel of $J_0(169) \ra A_f$ is connected. Since $\JacX$ 
is $\Q$-isogenous to $A_f$ (this was already used by Baran in~\cite{Bar14b}), it
suffices to show that $\mathrm{rk}(A_f/\Q) = 3$. The work of Gross-Zagier~\cite{GZ86} and
Kolyvagin-Logachev~\cite{KL89} proves that if the order of vanishing of the $L$-function
$L(f,s)$ of $f$ at $s=1$ is~1, then $\mathrm{rk}(A_f/\Q) = g= 3$.

\par We showed that $\mathrm{ord}_{s=1} L(f,s)> 0$  by computing the eigenvalue of $f$
under the Fricke involution $W_{169}$ and by computing the rational number
$c_{A_f}L(f,1)/\Omega_{A_f}$ exactly using the algorithm of~\cite{AS05}, where $c_{A_f}$
is the Manin constant of ${A_f}$ and $\Omega_{A_f}$ is the real period. So it only remains
to show that $L'(f,1) \ne 0$, which we did using~{\tt Magma}. We found that the number
$L'(f,1)$ was always larger than~0.6 for any embedding $\Q (\alpha )\hookrightarrow \mathbb{C}$ 
 and the error in these computations was smaller than $10^{-100}$. \end{proof}
\end{prop}

\begin{Remark}
An alternative approach was explained to us by Schoof~\cite{Sch12}. 
 The computation of the rank using descent is discussed by Bruin, Poonen and Stoll in
 \cite{BPS16}. They show that the rank is at least~3 by exhibiting three rational points
 in $\JacX$ which are independent modulo torsion. To carry out the descent needed to bound the rank from above, one needs to compute the class group of a certain number field
 $L$ of degree~28 and discriminant $2^{42}\cdot 13^{12}$. In \cite{BPS16}, this enables
 the authors to compute the rank assuming the Generalized Riemann Hypothesis. The authors
 of~\cite{BPS16} suggest that ``the truly dedicated enthusiast could probably verify unconditionally that the class group of $\mathcal{O}_L$ is trivial.'' 
\end{Remark}

\subsection{Semi-stable reduction of \texorpdfstring{$X_{\sC}(\ell)$}{X+sc(\ell)}}\label{sec:reduction}
We show that $X_{\sC}(13)$ has potentially good reduction by computing, more generally,
semi-stable models of the split Cartan modular curves $X_{\sC}(\ell)$ for primes $\ell\equiv 1
\pmod{12}$ over the integers in an explicit extension of $\Q_\ell$, using the work of
Edixhoven \cite{Edi89,Edi90}. 
For the remainder of this subsection, we let $\ell$ denote a prime number of the form
  $\ell = 12k+1$.

\begin{Remark}
For simplicity, we restrict to the case $\ell \equiv 1 \pmod{12}$, when there are no supersingular
  elliptic curves with $j$-invariant $0$ or $1728$. The same analysis would
  work in general if one in addition makes the action of the additional automorphisms
  explicit, which is done in \cite[Section 2.1.3]{Edi89}. We also note that additional
  level structure away from $\ell$ has little effect on our analysis, and may easily be included, mutatis mutandis. 

\par For a Dedekind domain $R$ we say $\phi:\mathcal{X} \rightarrow
  \mathrm{Spec} \ R$ is a \textit{model} for a smooth, proper, geometrically connected
  curve $X$ over the field of fractions of $R$ if $\phi$ is proper and flat, $\mathcal{X}$
  is integral and normal, and the generic fibre of $\mathcal{X}$ is isomorphic to $X$ over
  the base field. Such a model is called \textit{semi-stable} if all its geometric fibres
  are reduced and have at most ordinary double points as singularities. In what follows,
  we set $W$ to be the ring of Witt vectors over $\overline{\F}_\ell$, with field of
fractions $\Q_\ell^{\nr}$. 
Furthermore, we set $\mathrm{Ig}(\ell)$ to be the \textit{Igusa curve}, which is the
coarse moduli space over $\overline{\F}_\ell$ classifying elliptic curves $E \rightarrow
S/\overline{\F}_\ell$ together with $\Gamma_1(\ell)$-Drinfeld level structure on
$E^{(\ell)}$ which generates the kernel of the Verschiebung map $V: E^{(\ell)} \rightarrow E$, see \cite[Section 12.3]{KM85}.
\end{Remark} 

We start by recalling the work of Edixhoven on the semi-stable reduction of $X_0(\ell^2)$.
The description of the semi-stable model may be found in \cite{Edi90}, and the
statements about $w_{\ell^2}$ are in \cite[\S 2.2, 2.3.4]{Edi89}.
\begin{theorem}[Edixhoven]\label{Edix}
Then the curve $X_0(\ell^2)$ obtains semi-stable reduction over
  $\Q^{\nr}_\ell(\varpi)$, where $\varpi^{12k(6k+1)} = \ell$. Its special fibre consists of the following components:
\begin{itemize}
  \item $k$ \textbf{horizontal} components, one for each supersingular
    elliptic curve, all isomorphic to $u^2 = v^{\ell+1}+1$.
\item Four \textbf{vertical} components, of which two are rational, and two are isomorphic
  to $\mathrm{Ig}(\ell)/\pm 1$. 
\end{itemize}
Every horizontal component intersects every vertical component exactly once, and there are
  no other intersections. The Atkin--Lehner operator $w_{\ell^2}$ stabilises every
  horizontal component, and acts via $(u,v) \longmapsto (u,-v)$ in the above coordinates.
  Furthermore, $w_{\ell^2}$ permutes the rational vertical components, and stabilises the Igusa curves. 
\end{theorem}

We denote $\mathcal{X}$ for the semi-stable model of $X_0(\ell^2)$ over $W[\varpi]$
 constructed by Edixhoven, where $\varpi^{12k(6k+1)} = \ell$. We denote $\mathfrak{m}$ for the maximal ideal of $W[\varpi]$, and for any scheme $\mathcal{Y}$ over $W[\varpi]$, we denote $\mathcal{Y}_s$ for its special fibre.

\begin{theorem}\label{ThmSemi}
The curve $X_{\sC}(\ell)$ obtains semi-stable reduction over the field
  $\Q_\ell^{\nr}(\varpi)$, where $\varpi^{\ell^2-1} = \ell^2$. There exists a semi-stable
  model, which is stable unless $\ell=13$, whose special fibre consists of the following components:
\begin{itemize}
\item $k$ \textbf{horizontal} components, one for each supersingular elliptic curve, all isomorphic to $u^2 = v^{6k+1}+1$.
\item Three \textbf{vertical} components, of which one is rational, and two are isomorphic
  to $\mathrm{Ig}(\ell)/\mathrm{C}_4$. 
\end{itemize}
Every horizontal component intersects every vertical component exactly once, and there are no other intersections.
\begin{proof}
Because $X^+_0(\ell^2) \simeq_{\Q} X_{\sC}(\ell)$, it follows from \cite[Proposition 5]{Ray90}
  that the quotient $\mathcal{X}^+ = \mathcal{X}/w_{\ell^2}$ is a semi-stable model for
  $X_{\sC}(\ell)$. To describe the special fibre of this model, note that the order of
  $w_{\ell^2}$ is invertible on $\mathcal{O}_{\mathcal{X}}$, and we therefore have 
\[
\HH^1\left(\langle w_{\ell^2}\rangle,\mathfrak{m}\right) = 0.
\]
This implies that $\mathcal{O}_{\mathcal{X}^+}/ \mathfrak{m} \simeq \left(
\mathcal{O}_{\mathcal{X}} / \mathfrak{m} \right)^{w_{\ell^2}}$, and hence $\mathcal{X}^+_s =
\mathcal{X}_s/w_{\ell^2}$. The description of the special fibre follows from that of the
action of $w_{\ell^2}$ in Theorem \ref{Edix}. First, we note that both rational components of
$\mathcal{X}_s$ are identified by $w_{\ell^2}$, giving rise to a unique rational component in
the quotient. By~\cite[\S 2.2]{Edi89}, the quotients of the non-rational vertical
  components are isomorphic to the Igusa curves $\mathrm{Ig}(\ell)/\mathrm{C_4}$, which are of genus $(k-1)(3k-2)/2$, whereas the quotients of the horizontal components have affine equation $u^2 = v^{6k+1}+1$, and are hence of genus $3k$. The result follows.

\end{proof}
\end{theorem}

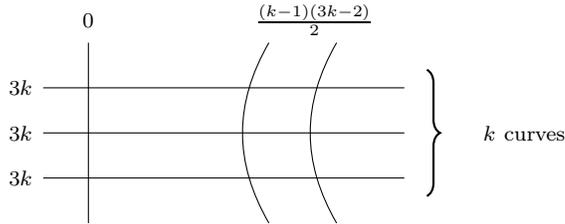
\begin{figure}[htbp!]\label{Fig}
\begin{center}
\begin{tikzpicture}[-,scale=.6]
 \path (-4,1) edge (4,1);
 \path (-4,0) edge (4,0);
 \path (-4,-1) edge (4,-1);

 \path (-3,2) edge (-3,-2); 
 \path (1,2) edge[bend right=30] (1,-2);
 \path (2.5,2) edge[bend right=30] (2.5,-2);
 
 \draw [decoration={brace,amplitude=0.5em},decorate, thick,black] (4.5,1.4) --  (4.5,-1.4);
 \node[text width=1.5cm] at (7,0) {\footnotesize $k$ curves};
 \node[] at (-3,2.5) {\footnotesize $0$};
 \node[] at (-4.5,1) {\footnotesize $3k$};
 \node[] at (-4.5,0) {\footnotesize $3k$};
 \node[] at (-4.5,-1) {\footnotesize $3k$};

 \node[text width=3cm] at (3.2,2.5) {\footnotesize $\frac{(k-1)(3k-2)}{2}$};
\end{tikzpicture}
\end{center}
\caption{Reduction of the semi-stable model of $X_{\sC}(\ell)$. }
\end{figure}

\par As a consequence, we obtain the genus formula $g = 6k^2 - 3k$ for the curve
$X_{\sC}(\ell)$. 

\begin{corollary}\label{cor:redn_13}
The split Cartan modular curve $X_{\sC}(13)$ has good reduction outside $13$, and potentially good reduction at $13$. More precisely, it obtains good reduction over the field $\Q_{13}^{\nr}(\varpi)$, where $\varpi^{84} = 13$. 
\begin{proof}
 The result follows from Theorem \ref{ThmSemi}. Indeed, 
we may contract all three rational curves to obtain a smooth model over $W[\varpi]$. 
\end{proof}
\end{corollary}

\begin{Remark}Since $X_{\ns}(13)\simeq X_{\sC}(13)$, the same result holds for the non-split Cartan modular curve of level~13. \end{Remark}

\subsection{Defining equations and known rational points }Baran \cite{Bar14a} finds an explicit defining equation
for $X_{\sC}(13)$ as follows. As the curve $X_{\sC}(13) \simeq
X_0^+(169)$ is of genus $3$, it is either hyperelliptic or has a smooth plane quartic
model. It may be
checked that the $q$-expansions of the Galois conjugates of $f$ do not satisfy a quadratic relation, but do satisfy
a quartic relation, resulting in the plane model\footnote{Hopefully, no confusion will arise from our use of the letters $X,Y$ and
    $Z$, which also denote other objects in this paper, as projective coordinates.}
\[ (-Y-Z)X^3 +(2Y^2 +ZY)X^2 +(-Y^3 +ZY^2 -2Z^2 Y+Z^3)X+(2Z^2 Y^2 -3Z^3 Y)=0,
\]
which has good reduction away from~13. To apply the algorithms of \cite{Tui16, Tui17}, it will be convenient to have a plane quartic model whose $Y^4$-coefficient is 1. For this reason we apply the substitution $(X:Y:Z)\mapsto (X-Y:X+Y:X+Z)$ 
giving the model $Q(X,Y,Z) = 0$, where
  \begin{align*}
    Q(X,Y,Z)=& \,Y^4 + 5X^4 - 6X^2 Y^2 + 6X^3 Z + 26X^2 YZ + 10XY^2 Z - 10Y^3 Z - 32X^2 Z^2
    -40XYZ^2\\ &+ 24Y^2 Z^2 + 32XZ^3 - 16YZ^3
  \end{align*}
  which we will use henceforth  and which has good reduction away from~2 and~13.
  With respect to this model, the 7 known rational points are as follows:
 \begin{center}
 \begin{tabular}{|c |c|c|c|c|c|c|}
   \hline $P_0$ &$P_1$ &$P_2$ &$P_3$ &$P_4$ &$P_5$ &$P_6$ \\
   \hline  $ (1 : 1 : 1)$&$ (1 : 1 : 2)$ & $ (0 : 0 : 1)$ & $ (-3 : 3 : 2)$ & $ (1 : 1 : 0)$ & $ (0 : 2 : 1)$ & $ (-1 : 1 : 0)$\\
   \hline
  \end{tabular}
  \end{center}
  In the remainder of this section, we show that there are no other rational points on $X$.

\subsection{Finding rational points on the first affine chart} 
\label{subsec:patch1}
Set $Y$ to be the affine chart $Z\neq 0$ with respect to the model $Q=0$, with coordinates $x=X/Z$, $y=Y/Z$. Then $Y$ contains all known rational points, except $P_4$ and $P_6$. Let us choose the basepoint to be $b= P_2 = (0,0)$. Define $Q_y =\partial Q/\partial y$, and consider the differentials
\[
\boldsymbol{\omega} := \left(
\begin{array}{c}1 \\ 
x \\y \\
-160x^4/3 + 736x^3/3-16x^2 y/3+436x^2/3-440xy/3+68y^2/3 \\ 
-80x^3/3+44x^2-40xy/3+68y^2/3- 32 \\
-16x^2 y+28x^2+72xy-4y^2-160x/3+272/3 \\
\end{array} \right) \frac{dx}{Q_y}. \]
which satisfy all the properties in \S \ref{subsec:notation}. By computing several small Hecke operators with respect to this basis, and taking appropriate linear combinations, we obtain matrices
\[
Z_1 =
\left(
\begin{array}{cccccc}
    0 & -976 & -1104 &  10 &  -6 &  18 \\
  976 &   0 & -816 &  -3 &  1 &   3 \\
 1104 &   816 &   0 &  -3 &   3 & -11 \\
  -10 &    3 &   3 &   0&   0 &   0 \\
    6 &  -1 &  -3 &    0 &   0 &  0 \\
  -18 &    -3 &    11 &    0 &   0 &   0 \\
\end{array} \right) 
\qquad 
Z_2 = 
\left(
\begin{array}{cccccc}
  0 &  112 & -656 &   -6 &   6 &   6 \\
-112 &     0 & -2576 & 15 &   9 &  27 \\
  656 & 2576 &   0 &  3 &  3 &  -3 \\
    6 & -15 &  -3 &  0 &  0 &  0 \\
   -6 &  -9 &  -3 &  0 &  0 &  0 \\
   -6 & -27 &   3 &   0 &   0 &  0 \\
\end{array} \right)
\]
which encode independent Tate classes $Z_1,Z_2 \in \HH^1 _{\dR}(X)\otimes \HH^1 _{\dR} (X)$ with respect to the basis $\boldsymbol{\omega}$, which satisfy the conditions of~\S\ref{subsec:bundle-AZ}.
 We find that a basis of $\HH^0 (X,\mathcal{O}(X\backslash Y))$ is given by $1,x,y,x^2,xy,y^2 $.
Using the algorithm outlined after Lemma~\ref{lemma:silly_RR}, we compute the Hodge filtration of the connections $\cA_{Z_i}$, and obtain:
\[
\begin{array}{lclclcl}
  \eta _{Z_1 } &=& -(44x^2 + 148/3xy + 8y^2 )\frac{dx}{Q_y} & & 
  \eta _{Z_2 } &=& -(40x^2+148xy+36y^2 )\frac{dx}{Q_y} \\
\boldsymbol{\beta}_{\Fil, Z_1 } &=& (0,1/2,1/2)^{\intercal} & & 
\boldsymbol{\beta}_{\Fil, Z_2 } &=& (0,-1/2,-5/2)^{\intercal} \\
\gamma_{\Fil,Z_1 } &=& 5y/6+3x/2 & & 
\gamma_{\Fil,Z_2 } &=& -5y/6-15x/2. \\
\end{array}
\]

\par Define $\mathcal{U}_1 :=Y_{\F_p} \cap \{ Q_y \neq 0 \}$. As our model is monic in $y$, we can apply the methods of \cite{Tui16} to define a lift $\Phi $ of Frobenius on a strict open neighbourhood of $]\mathcal{U}_1[$ satisfying $\Phi(x) = x^p$. The base point $b = P_2$ is a Teichm\"uller point with respect to our chosen $\Phi$. The Frobenius structure of $\cA_{Z_i}^{\rig}$ is determined using the techniques of~\S\ref{subsec:AZ-Frobenius}. This enables us to compute $\theta_{Z_1}$ and $\theta_{Z_2}$ as a power series on every residue disk in $\mathcal{U}_1$ via Lemma~\ref{lemma:ht_expl}. 


\subsubsection{Equivariant $p$-adic heights and determinants}
In this subsection, we explain how to compute the function in Lemma~\ref{lemma:qc_pair} for the
quadratic Chabauty pairs associated to $Z_1$ and $Z_2$, respectively. First, we can compute the action of $K = \Q(\zeta_7)^+$ on $\HH^1_{\dR}(X_{\Q_q})$ for any prime $q$ of good reduction by noting that, if $\iota $ denotes the inclusion of $K\otimes \Q _q $ into $\End (\HH^1 _{\dR} (X_{\Q _q }))$, we have 
\[
\iota (a_q )(f)=\Fr _q +q\Fr _q ^{-1}.
\]
Since $a_3 := a_{3}(f)=-\alpha  ^2 -2\alpha  $ generates $K$, we uniquely determine $\iota$ by computing $\iota(a_3)$. Starting from the splitting defined by the basis $(\omega_i)_i$, which is not $K$-equivariant in the sense of Remark~\ref{rk:ht_equiv}, we use $\iota$ to compute a $K$-equivariant splitting of the Hodge filtration on $\HH^1 _{\dR}(X)$. Let $h$ denote the $p$-adic height taken with respect to this splitting; it is $K$-equivariant in the sense of Remark~\ref{rk:ht_equiv}. Finally, as in the introduction we set
\[
\mathcal{E} = \HH^0(X_{\Q _p },\Omega ^1 )^* \otimes_{K\otimes \Q_p} \HH^0 (X_{\Q _p },\Omega ^1 )^*.
\]

\par By Lemma \ref{lemma:stating_obvious} and Lemma \ref{lemma:ht_qcpair}, we need to consider, for $x \in X(\Q_p)$ and $Z \in \{Z_1,Z_2\}$, the extensions
\[
E_1(x) := E_1(A_Z(b,x)), \qquad E_{2,Z}(x):= E_2(A_Z(b,x))
\]
with notation as in \eqref{E1E2}, viewed as elements of $\HH^0(X_{\Q_p}, \Omega ^1)^*$. We start by computing $E_{1}(P_i)$ and $E_{2,Z}(P_i)$ for our known points $P_0,\ldots,P_6 \in X(\Q)$. We find that $E_1(P_4)$ is nonzero, and hence generates $\HH^0 (X_{\Q _p }, \Omega ^1)^* $ over $K_p := K\otimes \Q _p$. Moreover, we compute that the elements
\[
\iota (a_3 )^i \left(E_{1 }(P_4 )\otimes _{K_p }E_{2,Z_1 }(P_4 ) \right), \qquad i = 0,1,2
\]
are a $\Q_p$-basis for $\mathcal{E}$, and we define $\psi_1 ,\psi_2 ,\psi_3$ to be the dual basis of $\mathcal{E}^*$. We also find that
\[
\begin{array}{llr}
E_1 (P_1 )\otimes_{K_p} E_{2,Z_1 }(P_1 )  &=&\iota(3164+1994\alpha + 294\alpha ^2) \cdot E_1 (P_4 )\otimes_{K_p} E_1 (P_4 ) \\
E_1 (P_3 )\otimes_{K_p} E_{2,Z_1 }(P_3 ) &=&\iota(1574 + 1006\alpha +150\alpha ^2) \cdot E_1 (P_4 )\otimes_{K_p } E_1 (P_4 ) \\
E_1 (P_4 )\otimes_{K_p} E_{2,Z_1 }(P_4 ) &=&\iota(-232-134 \alpha -18\alpha ^2) \cdot E_1 (P_4 )\otimes_{K_p } E_1 (P_4 )
\end{array}
\]
so that the three classes on the left form a basis for $\mathcal{E}$. By Lemma~\ref{lemma:qc_pair} and~\ref{lemma:ht_qcpair}, this gives two matrices
\[
 T_i(x) := 
 \left(\begin{array}{cccc}
 \theta_{Z_i}(x) & \Psi_1(Z_i\ ,x) & \Psi_2(Z_i\ ,x) & \Psi_3(Z_i\ ,x) \\
 \theta_{Z_1}(P_1) & \Psi_1(Z_1,P_1) & \Psi_2(Z_1,P_1) & \Psi_3(Z_1,P_1) \\
 \theta_{Z_1}(P_3) & \Psi_1(Z_1,P_3) & \Psi_2(Z_1,P_3) & \Psi_3(Z_1,P_3) \\
 \theta_{Z_1}(P_4) & \Psi_1(Z_1,P_4) & \Psi_2(Z_1,P_4) & \Psi_3(Z_1,P_4) 
 \end{array} \right),  
 \qquad \Psi_j(Z,z) := \psi_j(E_1(z)\otimes_{K_p} E_{2,Z}(z)),
\]
such that the locally analytic functions $\det(T_i(x)) : X(\Q _p ) \to \Q_p $ vanish on $X(\Q _p )_2$. It only remains to determine their common zeroes. This may be done by computing up to high enough precision, in the following sense. Let $F_i \in \Z_p \llbracket t \rrbracket$, and suppose
\[
G(t)=\sum _{n\geq 0}c_n t^n =\sum a_{ij}\int F_i (\int F_j )+\sum a_i \int F_i +\sum b_i F_i 
\] 
is a $\Q_p$-linear combination of the $F_i$, their single integrals and their double integrals. Then an elementary estimate gives us 
\[
v_p (c_n )\geq \min (\{ v_p (a_i )\} \cup \{ v_p (a_{ij})\} \cup \{ v_p (b_i )\})-2 \lfloor \log _p (n) \rfloor,
\] 
so that if we compute enough coefficients for the power series $F_i$, the slopes of the Newton polygon of $G(t)$ beyond our precision are bounded below by $-1$, and can hence not come from $\Q_p$-rational points.

\par The following table contains the zeroes of the functions $\det(T_i(x))$on $]\mathcal{U}_1
[$, computed to precision $O(17^5)$; all of
them are simple.
\begin{center}
\begin{small}
 \begin{tabular}{|c||r|r|}
    \hline
$X (\F_{17})$  & $\det(T_1 (x))=0$ & $\det(T_2(x))=0$  \\
  \hline
  \hline
  $(0,0)$ & ${\bf 0}$ & ${\bf 0}$ \\ 
  $(0,2)$  & ${\bf 0}$ & ${\bf 0}$ \\
&   $9 \cdot 17 + 11 \cdot 17^2 + 11 \cdot 17^3 $ & \\
  $(4,1)$  & & $4 + 13 \cdot 17 + 11 \cdot 17^2 + 4 \cdot 17^3 + 12 \cdot 17^4 $ \\
& & $4 + 15 \cdot 17 + 4 \cdot 17^2 + 2 \cdot 17^3 + 11 \cdot 17^4 $  \\
  $(5,-2)$ & $5 + 14 \cdot 17 + 7 \cdot 17^2 + 13 \cdot 17^3 + 12 \cdot 17^4 $ &   \\
& & $5 + 15 \cdot 17 + 10 \cdot 17^2 + 8 \cdot 17^3 + 8 \cdot 17^4$ \\
& & $5 + 3 \cdot 17 + 8 \cdot  17^2 + 4 \cdot 17^3 + 2 \cdot 17^4 $ \\
   $(7,-7)$  & ${\bf 7 + 8 \cdot 17 + 8 \cdot 17^2 + 8 \cdot 17^3 + 8 \cdot 17^4}$ & ${\bf 7 + 8 \cdot 17 + 8 \cdot 17^2 + 8 \cdot 17^3 + 8 \cdot 17^4}$   \\ 
            & $7 + 6 \cdot 17 + 10 \cdot 17^2 + 14 \cdot 17^3 + 15 \cdot 17^4$ & \\
& & $7 + 7 \cdot 17 + 8 \cdot 17^3 + 7 \cdot 17^4 $  \\
  $ (7,6) $ & $7 + 8 \cdot 17 + 13 \cdot 17^2 + 3 \cdot 17^3 + 3 \cdot 17^4  $ & \\
  $ (8,0)$ & $8 + 8 \cdot 17 + 12 \cdot 17^2 + 16 \cdot 17^3 + 7 \cdot 17^4$ &\\
&  $8 + 16 \cdot 17 + 16 \cdot 17^2 + 17^3 + 9 \cdot 17^4 $ & \\
& & $8 + 10 \cdot 17 + 7 \cdot 17^2 + 2 \cdot 17^3 + 10 \cdot 17^4 $ \\
& & $ 8 + 5 \cdot 17 + 16 \cdot 17^2 + 6 \cdot 17^3 + 12 \cdot 17^4 $ \\
  $(8,-14)$  & $8 + 8 \cdot 17 + 14 \cdot 17^2 + 8 \cdot 17^3 + 16 \cdot 17^4$ &  \\
&  $8 + 10 \cdot 17 + 12 \cdot 17^2 + 2 \cdot 17^3 + 8 \cdot 17^4 $ & \\
  $(9,-4)$  & $9 + 10 \cdot 17 + 8 \cdot 17^2 + 4 \cdot 17^3 + 6 \cdot 17^4 $  &  \\
  $(9,-8)$  & ${\bf 9 + 8 \cdot 17 + 8 \cdot 17^2 + 8 \cdot 17^3 + 8 \cdot 17^4 }$  & ${\bf 9 + 8 \cdot 17 +
8 \cdot 17^2 + 8 \cdot 17^3 + 8 \cdot 17^4 }$  \\
& $9 + 6 \cdot 17^2 + 6 \cdot 17^3 + 3 \cdot 17^4 $ & \\
  $(13,-8)$  & $13 + 12 \cdot 17 + 7 \cdot 17^2 + 2 \cdot 17^3 + 12 \cdot 17^4 $ & \\
&  $13 + 13 \cdot 17 + 7 \cdot 17^2 + 7 \cdot 17^3 + 13 \cdot 17^4 $  &  \\
& & $13 + 4 \cdot 17 + 10 \cdot 17^2 + 3 \cdot 17^3 + 7 \cdot 17^4 $  \\
$(15,-3) $  & $15 + 8 \cdot 17 + 9 \cdot 17^2 + 4 \cdot 17^3 + 16 \cdot 17^4 $ &  \\
& $15 + 4 \cdot 17 + 5 \cdot 17^2 + 2 \cdot 17^4 $ & \\
\hline 
\end{tabular}
\end{small}
\end{center}
This recovers the points $P_2$, $P_5$, $P_3$ and $P_1$, and shows that there are no
other $\Q$-rational points in $]\mathcal{U}_1 [$. 
\begin{Remark}
Note that by its very construction, the function $\det(T_1(x))$ vanishes on the rational points $P_1,P_3$. However, the rational points $P_2$ and $P_5$ were not used as input for its construction, so that the vanishing of $\det T_1(x)$ on these two points seems to us an extremely convincing confirmation of the correctness of our algorithms. Furthermore, the vanishing of $\det(T_2(x))$ at \textit{none} of the four points $P_1,P_2,P_3,P_5$ is automatic, and again the definition of this function does not take $P_2$ or $P_5$ as input. 
\end{Remark}

\subsection{Rational points on $]\mathcal{U}_2 [$}\label{sec:u2}
\label{subsec:patch2}
We now consider a second affine chart $Y'$ defined by $X\neq 0$ with respect to the model $Q=0$, with coordinates $u:=Z/X$ and $v:=Y/X$. Then $Y'$ contains all known rational points, except $P_2$ and $P_5$. Let us choose the basepoint to be $b = P_6 = (0,-1)$. Define $Q_v =\partial Q/\partial v$, and consider the differentials
\[
\boldsymbol{\omega}' := \left( \begin{array}{c}
\\ -u \\ -1 \\ -v \\ 
\frac{768}{5}u^2v - \frac{448}{5}uv^2 - \frac{1536}{5}u^2 + 96uv + 16v^2 + \frac{2272}{15}u - \frac{1648}{15}v + \frac{1712}{15} \\ 
\frac{128}{7}u^2v^2 - \frac{5056}{35}u^2v + \frac{576}{35}uv^2 + \frac{7552}{35}u^2 - \frac{816}{7}uv + \frac{136}{7}v^2 + \frac{10736}{105}u - \frac{1072}{15}v - \frac{184}{105} \\ -\frac{448}{5}u^2v + \frac{288}{5}uv^2 + \frac{896}{5}u^2 - 80uv - 8v^2 - \frac{2272}{15}u + \frac{96}{5}v - \frac{1432}{15} \\
\end{array} \right) du/Q_v
\]
which satisfy all the properties in \S \ref{subsec:notation} with respect to $Y'$. Furthermore, they are cohomologous to the previous differentials $\boldsymbol{\omega}$ in $\HH^1_{\dR}(X)$, so that the matrices $Z_1$ and $Z_2$ remain unchanged.

\par We calculated the Hodge filtration using the algorithm outlined after Lemma~\ref{lemma:silly_RR}. To compute the Frobenius structure, define $\mathcal{U}_2 :=Y'_{\F_p} \cap \{ Q_v \neq 0 \}$. As our model is monic in $v$, we can again apply the methods of \cite{Tui16} to define a lift $\Phi$ of Frobenius on a strict open neighbourhood of $]\mathcal{U}_2[$ satisfying $\Phi(u) = u^p$. The base point $b = P_6$ is a Teichm\"uller point with respect to our chosen $\Phi$. The Frobenius structure of $\cA_{Z_i}^{\rig}$ is determined using the techniques of~\S\ref{subsec:AZ-Frobenius}. This enables us to compute $\theta_{Z_1}$ and $\theta_{Z_2}$ as a power series on every residue disk in $\mathcal{U}_2$ via Lemma~\ref{lemma:ht_expl}. 

\par Using the same rational points $P_1,P_3,P_4$ as we did in the previous section, we
construct two matrices $T_i'(u)$, whose determinant vanishes at all the rational points.
It suffices to check the residue disks of $(1:1:0)$, $(1:-1:0)$ and $(1:1:1)$. The
Frobenius lift we chose was not defined on the residue disk of $(1:1:1)$, but for the
other two disks we obtain the following zeroes to precision $O(17^5)$.
\begin{center}
\begin{small}
 \begin{tabular}{|c || r|r|}
    \hline
$X (\F_{17})$  & $\det T_1'(u)=0$& $\det T_2' (u)=0$ \\
  \hline
  \hline
$(0,-1)$ & {\bf 0} & {\bf 0}\\ 
 & $12\cdot 7 + 5\cdot 17^2 + 2\cdot 17^3 + 9\cdot 17^4 $ & \\
 & & $12\cdot 7 + 4\cdot 17^2 + 14\cdot 17^3 + 12\cdot 17^4$ \\
$(0,1)$ & {\bf 0} & {\bf 0} \\
 & $8\cdot 7 + 7\cdot 17^3 + 7\cdot 17^4 $ & \\
 & & $2\cdot 7 + 3\cdot 17^2 + 6\cdot 17^3 + 16\cdot 17^4 $ \\
\hline 
\end{tabular}
\end{small}
\end{center}
Again all the zeroes are simple, and we recover the points $P_6$ and $P_4$ as common zeroes of both functions. Combined with the calculations of \S \ref{subsec:patch1}, we have shown that there are no points on $X(\Q)$ besides the known ones, except perhaps on the residue disk of $(1:1:1)$.

\subsection{Rational points on $](1:1:1)[$}\label{sec:leftover}
The remaining residue disk lies at the point $P_0=(1:1:1)$, which was the disk where the
Frobenius lift above is not defined. Rather than choosing a new lift of Frobenius, as explained in ~\S\ref{trick} we may use Lemma \ref{lemma:total_complete_utter_nightmare} to reduce the computation of $p$-adic heights of $A_Z(P_0 ,x)$, for $x$ in $]P_0 [$, to the problem of computing the single integrals $\int^{P_0 }_{b}\boldsymbol{\omega }$. These integrals are computed using the original Frobenius lift, via overconvergence and evaluating at points defined over highly ramified extensions of $\Q _p $ (see~\cite[Prop 3.8, Prop 4.3]{BT}). We find that, for both choices of $Z$, the only roots of the resulting power series are at $P_0 = (1:1:1)$ (and are simple). This completes the proof of Theorem~\ref{thm:main_thm}.

\appendix

\section{Background on universal objects, isocrystals and $p$-adic comparison}
\label{sec:Tannakian}

In this appendix, we briefly discuss the notion of universal objects in unipotent Tannakian categories, and discuss them in the three examples of importance to us: The categories of unipotent $\Q_p$-lisse sheaves on $\overline{X}$, unipotent vector bundles with connection on $X_{\Q_p}$, and unipotent isocrystals on $\cX_{\F_p}$. There are comparison isomorphisms between the universal objects in these categories, which we recall in \ref{subsec:nonabeliancomparison}. 

\subsection{Universal objects in unipotent Tannakian categories. }\label{subsec:Tannakian}

For a general unipotent neutral Tannakian category $\cC$ with fibre functors $\omega$ and $\nu$, we will first define for every $n\geq 1$ certain universal objects $\cA_n (\cC,\omega )$. Their utility comes from the fact that one can often compute `extra structure' on fundamental groups and path torsors by instead computing that extra structure on certain universal objects in the category. 

\par It is instructive to first consider the case of fundamental groups of topological spaces. If $X$ is a locally path connected topological space, with universal cover $\widetilde{X}$, then there is a well known correspondence between deck transformations of $\widetilde{X}$ and elements of the fundamental group. This is perhaps most naturally formulated by replacing the universal cover with a \textit{pointed} universal cover 
\[
p :(\widetilde{X},\widetilde{b}) \ \lra \  (X,b).
\]
Then the correspondence is given by the statement that the following map is bijective: 
\[
\pi _1 (X,b) \ \lra p^{-1}( \{ b \} ), \quad \gamma \longmapsto \gamma (\widetilde{b}).
\]

\subsubsection{Universal objects. }Similar universal objects may be constructed in certain Tannakian categories. For the main definitions on Tannakian categories and their fundamental groups, we refer to Deligne \cite{Del90}, and will make free use of the language introduced there. 

\begin{Definition} We say a neutral Tannakian category $\cC$ over a field $K$ with fibre functor $\omega$ is \textit{unipotent} if its fundamental group $\pi _1 (\cC,\omega)$ is pro-unipotent. Equivalently, $\cC$ is unipotent if every object $V\in \cC$ admits a nonzero morphism $\mathbf{1} \lra V$ from the unit object $\mathbf{1}$ in $\cC$. 
\end{Definition}

\par Let $\cC$ be a neutral unipotent Tannakian category over a field $K$ of characteristic zero, with fibre functor $\omega $, let $A(\cC,\omega )$ denote its pro-universal enveloping algebra, with augmentation ideal $I$, and define $A_n (\cC,\omega ):=A(\cC,\omega )/I^{n+1}$. Recall that there is a canonical isomorphism
\[
\varinjlim_n A_n (\cC,\omega )^* \ \stackrel{\sim}{\lra} \ \mathcal{O}(\pi _1 (\cC,\omega ))
\]
between the dual Hopf algebra and the co-ordinate ring of the affine group scheme $\pi _1 (\cC,\omega )$. Since $A_n (\cC;\omega )$ is a finite dimensional $K$-representation of $\pi _1 (\cC,\omega )$, it corresponds by Tannaka duality to an object $\cA_n (\cC,\omega )$ of the category $\cC$, with the property that $\omega (\cA_n (\cC,\omega ))=A_n (\cC,\omega )$.

\par Now suppose $(\cC,\omega )$ has finite dimensional Ext-groups. A \textit{pointed object} in $\cC$ is a pair $(V,v)$ where $V\in \cC$ and $v\in \omega (V)$. An object of $\cC$ is \textit{n-unipotent} if there exists a filtration 
\[
V=V_0 \supset \ldots \supset V_n
\]
by subobjects such that $V_i /V_{i+1}$ is zero or is isomorphic to a direct sum of copies of the trivial object, for all $i$. A morphism between pointed $n$-unipotent objects is a morphism in $\cC$ that respects the filtrations $V_i$, and the chosen vectors $v$. 

\begin{Definition}
We say a pointed $n$-unipotent object  $(\mathcal{E},e)$ is a \textit{universal} pointed $n$-unipotent object if for all pointed $n$-unipotent objects $(\mathcal{V},v)$ there exists a morphism of pointed $n$-unipotent objects 
\[
(\mathcal{E},e) \ \lra \ (\mathcal{V},v).
\]
Finally, a \textit{universal pointed pro-object} in $\cC$ is a compatible sequence $\left\{ (\mathcal{E}_n ,e_n ) \right\}_{n \geq 1}$ of universal pointed $n$-unipotent objects in $\cC$, equipped with maps of pointed objects 
\[
(\mathcal{E}_n ,e_n ) \ \lra \ (\mathcal{E}_{n-1},e_{n-1}).
\]
\end{Definition}

\par Note that if a universal pointed $n$-unipotent object exists, it is unique up to unique isomorphism. Since we have a canonical identification of $\omega (\cA_n (\cC,\omega ))$ with $A_n (\cC,\omega )$, we have an associated $n$-unipotent pointed object $(\cA_n (\cC,\omega ),1)$. Furthermore, the quotient map $A_{n+1}(\cC,\omega) \to A_{n}(\cC,\omega)$ induces transition maps
\[
\left(\cA_n(\cC,\omega ), 1 \right) \ \lra \ \left( \cA_{n-1}(\cC,\omega ),1 \right).
\]
From the equivalence between representation of $\pi _1 (\cC,\omega )$ and left $A(\cC,\omega )$-modules, we obtain:
\begin{lemma}
The inverse system $\left\{ (\cA_n (\cC,\omega ),1) \right\}_{n\geq 1} $ is a universal pointed pro-object in $\cC$.
\end{lemma}

\subsubsection{Path torsors. }
\label{subsec:universal-objects}
As in the topological case, we can define path torsors of the universal objects $\cA_n(\cC,\omega)$ in unipotent neutral Tannakian categories. If $\nu $ is another fibre functor of $\cC$, then recall there are corresponding path torsors 
\[
\pi _1 (\cC;\omega ,\nu )
\]
for the Tannakian fundamental group, given by the tensor compatible isomorphisms between $\omega$ and $\nu$. We define likewise
\[
A_n(\cC; \omega, \nu) := A_n (\cC,\omega ) \times _{\pi _1 (\cC,\omega )}  \pi _1 (\cC;\omega ,\nu ) 
\]
where the product is interpreted in the following sense: The co-ordinate ring $\mathcal{O}(\pi _1 (\cC;\omega ,\nu ))$ has the structure of a free $\mathcal{O}(\pi _1 (\cC,\omega ))$-comodule of rank one, giving $\mathcal{O}(\pi _1 (\cC ;\omega ,\nu ))^* $ the structure of a free $\mathcal{O}(\pi _1 (\cC ,\omega ))^* $-module of rank 1 hence we may define 
\[
A_n (\cC;\omega ,\nu ) :=\left((\mathcal{O}(\pi_1 (\cC;\omega ,\nu ))^* \otimes _{\mathcal{O}(\pi_1 (\cC;\omega ))^* }A_n (\cC,\omega )^* \right)^*.
\]

\par In the topological setting, the universal pointed cover has the following useful property, for any point $x\in X$, there is a canonical isomorphism 
\[
\pi _1 (X;b,x) \simeq p^{-1}({x}).
\] 
In the case of a neutral unipotent Tannakian category we have the following analogue, see for instance Kim \cite[\S 1]{Kim09} or Betts \cite[\S 6.2.2]{Bet17}.
\begin{lemma}\label{lemma:fibre}
Let $\omega $ and $\nu $ be fibre functors, and let $\omega_n $ and $\nu_n $ denote their restriction to the full subcategory of $n$-unipotent objects. Then we have functorial isomorphisms
\[
\nu (\cA_n (\cC,\omega ))\simeq A_n (\cC;\omega ,\nu ) \simeq \omega (\cA_n (\cC ,\omega )).
\]
\begin{proof}
By the universal property of $\cA_n $, the map
\begin{align*}
\Hom (\omega _n ,\nu _n ) \ & \lra \ \nu (\cA_n (\cC,\omega )) \\
F \ & \longmapsto \ F(\cA_n )(e_n )
\end{align*}
is an isomorphism of $K$-vector spaces. For the other identification, using $\Hom (\omega _n ,\omega _n )=A_n (\cC,\omega )$ gives a map
\[
\pi _1 (\cC;\omega ,\nu )\times _{\pi _1 (\cC,\omega )}A_n (\cC,\omega ) \ \lra \ \Hom (\omega _n ,\nu _n ) \ : \ (\gamma ,x) \ \longmapsto \ \gamma \circ x.
\]
Since both sides are free $A_ n(\cC,\omega )$-modules of rank one, this is an isomorphism.

\par For the second isomorphism, note that by duality we have an isomorphism $\Hom (\omega _n ,\nu _n )\simeq \Hom (\nu _n ,\omega _n )$, (i.e. the isomorphism is defined by sending $f\in \Hom (\omega _n ,\nu _n )$ to the morphism of functors $f^* $ sending $V\in \mathcal{C}$ to $f^* (V):=(f(V))^* $).
\end{proof}
\end{lemma}

The identification of $A_n (\cC  ;\omega ,\nu )$ with $\Hom (\omega _n ,\nu _n )$ induces composition maps
\[
A_n (\cC ;\omega _2 ,\omega _3 )\times A_n (\cC ;\omega _1 ,\omega _2 )\to A_n (\cC ;\omega _1 ,\omega _3 )
\]
for all fibre functors $\omega _1 ,\omega _2 ,\omega _3 $.
We may also describe the right action of $A_n (\cC ,\omega )$ on $\nu (\cA _n ,\omega )$ induced by these isomorphisms. Given $x\in \nu (\cA _n ,\omega )$, and $y\in A_n (\cC ;\omega \nu )$, the product $y.x$ is defined as follows. Take $\widetilde{x}$ to be the unique endomorphism $\cA _n (\omega )$ such that $\nu \widetilde{x}(e_n )=x$. Then 
\[ 
y.x=\widetilde{x}(y).
\] 
\subsection{Main examples. }\label{subsec:pathz}

We will now briefly discuss the main examples to which the above constructions are applied. Let us adopt the notation of \S~\ref{subsec:notation}, and consider the following neutral Tannakian categories:

\subsubsection{Unipotent lisse $\Q_p$-sheaves on $\overline{X}$. }
\label{subsec:background_lisse}

Let $X$ be as in \S~\ref{subsec:notation}, and let $K$ be a field of characteristic $0$. Consider the category of unipotent $\Q_p$-lisse sheaves on $\overline{X} := X \times_{\Q} \overline{K}$, which we will call $\cC^{\et}(\overline{X})$. Our choice of base point determines a geometric point $\overline{b} \in X(\overline{K})$, and hence a fibre functor $\overline{b}^*$ which makes $\cC^{\et}(\overline{X})$ into a neutral Tannakian category. For any other geometric point $\overline{x} \in X(\overline{K})$, we obtain a fibre functor $\overline{x}^*$. 

\par In the body of this paper, the Tannakian fundamental group is denoted by
\[
\pi_1^{\et,\scriptstyle{\Q_p}}(b) := \pi_1(\cC^{\et}(\overline{X});\overline{b}^*),
\]
which is a pro-unipotent affine group scheme over $\Q_p$, with an action of $G_K = \mathrm{Gal}(\overline{K}/K)$. The maximal $n$-unipotent quotient and its path torsors are denoted $\U_n^{\et}(b)$ and $\U_n^{\et}(b,x)$. We also use the notation:
\[
\left\{
\begin{array}{lll}
A_n^{\et}(b) &:=& A_n (\cC^{\et}(\overline{X});\overline{b}^*), \\
A_n ^{\et}(b,x) &:=& A_n (\cC^{\et}(\overline{X});\overline{b}^*,\overline{x}^* ).\\
\end{array} 
\right.
\]
as well as the notation $\cA_n^{\et}(b)$ for the corresponding universal $n$-unipotent object.

\subsubsection{Unipotent connections on $X_{K}$. }
\label{subsec:background_connections}

Let $X$ be as in \S~\ref{subsec:notation}, and let $K$ be a field of characteristic $0$. Consider the category of unipotent vector bundles with connection on $X_{K}$, which we will denote by $\cC^{\dR}(X_{K})$. Our choice of base point determines a fibre functor $b^*$ which makes $\cC^{\dR}(X_{K})$ into a neutral Tannakian category. For any point $x \in X(K)$ we obtain a fibre functor $x^*$. 

\par We denote its fundamental group by 
\[
\pi_1^{\dR}(X_{\Q_p},b) = \pi_1(\cC^{\dR}(X_{K});b^*),
\]
and the maximal $n$-unipotent quotient and its path torsors are denoted $\U_n^{\dR}(b)$ and $\U_n^{\dR}(b,x)$. We also use the following notation:
\[
\left\{
\begin{array}{lll}
A_n^{\dR}(b) &:=& A_n (\cC^{\dR}(X_{K});b^*), \\
A_n ^{\dR}(b,x) &:=& A_n (\cC^{\dR}(X_{K}); b^*, x^* ),
\end{array} 
\right.
\]
as well as the notation $\cA_n^{\dR}(b)$ for the corresponding universal $n$-unipotent object.

\subsubsection{Unipotent isocrystals on $\cX_{\F _p}$. }
\label{subsec:background_isocrystals}
We now recall some foundational results about the category of unipotent isocrystals on a curve over $\F _p $ \cite{Ber96,CLS99}. We first recall the notion of a rigid triple, and then define the category $\cC^{\rig}(\cX_{\F_p})$. 

\par We start by recalling the notion of a rigid triple. Related notions are those of a \textit{triple} in \cite{CT03}, or a $\Q_p $-\textit{frame} in \cite{LS07}. A \textit{rigid triple} over $\F_p$ is a triple $(Y,X,P)$, where 
\begin{itemize}
\item $P$  is a formal $p$-adic $\Z _p $ scheme, 
\item $X$ is a closed $\F _p $-subscheme of $P$, proper over $\F _p$, 
\item $Y \subset X$ is an open $\F _p$-subscheme such that $P$ is smooth in a neighbourhood of $Y$. 
\end{itemize}
Given a rigid triple $(Y,X,P)$, we let $P_{\Q_p}$ denote the Raynaud generic fibre of $P$. We let 
\[
]Y[ \ \subset P_{\Q _p }
\]
denote the tube of $Y$, which consists of all the points that reduce to a point of $Y$. Finally, let $j^\dagger \mathcal{O}_Y$ be the overconvergent structure sheaf on $]Y[$, as defined in \cite[\S 2.1.1.3]{Ber96}.

\begin{Definition}
Let $T = (Y,X,P)$ be a rigid triple. An overconvergent isocrystal on $T$ is a locally free $j^\dagger \mathcal{O}_Y$-module with connection. 
\end{Definition}

\par Given rigid triples $T=(Y,X,P)$ and $T'=(Y',X',P')$, and a morphism $f:Y\to Y'$, a \textit{compatible morphism} $T\to T'$ is a morphism 
\[
g:\mathcal{W} \ \lra \ P' _{\Q _p }
\]
from a strict neighbourhood of $]Y[$ to $P'_{\Q_p}$, which is compatible with $f$ via the specialisation map. 

\par Given two rigid triples $T=(Y,X,P)$ and $T'=(Y,X,P')$, there is a canonical
equivalence between the category of overconvergent isocrystals on $T$ and $T'$, via the
category of overconvergent isocrystals on $(Y,X,P\times _{\Z _p }P')$ , (see \cite[\S
2.3.1]{Ber96} or \cite[\S 7.3.11]{LS07}). For this reason we often suppress the choice of
rigid triple from our notation and terminology, and denote the category of unipotent
isocrystals on $Y$ by $\cC^{\rig}(Y)$. The category $\mathcal{C}^{\rig }(T)$ is
sometimes referred to as a \textit{realisation} of $\cC ^{\rig }(Y)$. By functoriality, for any $y \in Y(\F _p )$, we obtain a functor functor $y^* $ from $\cC^{\rig }(Y)$ to the category $\cC^{\rig}(\spec \F_p)$ of unipotent isocrystals on $\spec \F _p $, which  is canonically identified with the category of $\Q_p$-vector spaces, via the realisation given by the rigid triple 
\[
T = (\spec \F_p, \spec \F_p, \spf \Z_p).
\]
In this way $y^*$ can be viewed as a fibre functor on the unipotent Tannakian category $\mathcal{C}^{\rig }(Y_{\F _p })$. 

\par An explicit description of the fibre functor $y^* $ may be given as follows. Choose a lift $\widetilde{y}$ of $y$ to $]Y[$. Then $\widetilde{y}$ defines a fibre functor on $\cC ^{\rig }(T)$ in the obvious way. Whenever we write the fibre functor $y^* $ in this paper we shall mean $\widetilde{y}^*$ for some choice of $\widetilde{y}$.
The justification for this notation is that if $\widetilde{y}_1 $ and $\widetilde{y}_2 $ are two different lifts, then there is an isomorphism of functors
\[
T_{\widetilde{y}_1 ,\widetilde{y}_2 }:\widetilde{y}_1 \to \widetilde{y}_2
\]
defined by parallel transport as follows. Given an overconvergent isocrystal $(\mathcal{V},\nabla )$ on $T$, the pullback of $(\mathcal{V},\nabla )$ to $]y[$, and hence the maps
\[
\mathcal{V}(]y[)^{\nabla =0}\stackrel{\simeq }{\longrightarrow } \widetilde{y}_i ^* \mathcal{V}
\]
are bijective. The natural transformation $T_{\widetilde{y}_1 ,\widetilde{y}_2 }(\mathcal{V},\nabla )$ is defined as the composite 
\[
\widetilde{y}_1 ^* \mathcal{V} \stackrel{\simeq }{\longleftarrow }\mathcal{V}(]y[)^{\nabla =0}\stackrel{\simeq }{\longrightarrow } \widetilde{y}_2 ^* \mathcal{V}.
\]

\par The main example of interest to us is the following. Using the notation of \S~\ref{subsec:notation}, we denote $\cC^{\rig }(\cX_{\F_p })$ for the Tannakian category of unipotent isocrystals on the rigid triple 
\[
T = (\cX_{\F_p}, \cX_{\F_p}, \mathfrak{X}),
\]
where $\mathfrak{X}$ is the completion of $\cX$ along its special fibre. This will usually be referred to simply as the category of unipotent isocrystals on $\cX_{\F_p}$. Its fundamental group will be denoted by 
\[
\pi_1^{\rig }(\cX_{\F _p },\overline{b}) := \pi_1(\cC^{\rig}(\cX_{\F _p });\overline{b})
\]
and the maximal $n$-unipotent quotient and its path torsors are denoted $\U_n^{\rig}(\overline{b})$ and $\U_n^{\rig}(\overline{b},\overline{x})$. We also use the notation:
\[
\left\{
\begin{array}{lll}
A_n^{\rig}(\overline{b}) &:=& A_n (\cC^{\rig}(\cX_{\F_p} );\overline{b}^*), \\
A_n ^{\rig}(\overline{b},\overline{x}) &:=& A_n (\cC^{\rig}(\cX_{\F_p});\overline{b}^*,\overline{x}^* ).\\
\end{array} 
\right.
\]
as well as the notation $\cA_n^{\rig}(\overline{b})$ for the corresponding universal $n$-unipotent object.  When we want to emphasise the choice of a rigid triple $(Y,X,P)$, we write $A_n ^{\rig }(b,x)$ and $A_n ^{\rig }(b)$, where $b,x$ are $\Q _p $ points of $P_{\Q _p }$ lying above $\overline{b}$ and $\overline{x}$ respectively.

\par Pull-back by absolute Frobenius induces an endofunctor on $\cC^{\rig }(\cX_{\F_p})$, giving a Frobenius action on $\pi_1^{\rig}(X_{\F _p };b,x)$, see \cite[\S 2.4.2]{CLS99}. This induces an action on universal objects $\cA _n ^{\rig }(b)$. Transporting this to the realisation $\cC ^{\rig }(Y,X,\mathfrak{X})$, we obtain, for any choice of Frobenius lift $\phi $, an $F$-isocrystal $\cA _n ^{\rig }$ on $(Y,X,\mathfrak{X})$. Let $b_0$ be a Teichm{\"u}ller lift of $b$ with respect to the Frobenius lift $\phi $.

\begin{lemma}\label{lemma:too_far}
There is a morphism $\Phi :\phi ^* (\mathcal{A}_n ^{\rig }(b )) \to \mathcal{A}_n ^{\rig }(b )$ of overconvergent isocrystals on $(Y,X,\mathfrak{X})$ such that, for all points $z\in Y(\F _p )$ with Teichm\"uller representative $z_0 $, we have a commutative diagram
\[
\begin{tikzpicture}
\matrix (m) [matrix of math nodes, row sep=3em,
column sep=3em, text height=1.5ex, text depth=0.25ex]
{\varprojlim z_0 ^* (\cA _n ^{\rig },\nabla )(\overline{b}) & \varprojlim z_0 ^* (\cA _n ^{\rig} ,\nabla )(\overline{b})  \\
 \Hom (\overline{b}^* ,\overline{z}^* ) & \Hom (\overline{b}^* ,\overline{z}^* ) \\};
\path[->]
(m-1-1) edge[auto] node[auto]{} (m-2-1)
edge[auto] node[auto] { $z_0 ^* \Phi $ } (m-1-2)
(m-2-1) edge[auto] node[auto] { $\phi ^* $} (m-2-2)
(m-1-2) edge[auto] node[auto] {} (m-2-2);
\end{tikzpicture}
\]
\end{lemma}
\begin{proof}
This follows from the functoriality of the isomorphism in Lemma \ref{lemma:fibre}.
\end{proof}
\subsection{Non-abelian comparison theorems}
\label{subsec:nonabeliancomparison}

In this subsection we recall some comparison theorems relating universal objects in the three unipotent Tannakian categories defined above. Resume the notation of \S~\ref{subsec:notation} and set $\mathcal{X}$ to be a smooth projective curve over $\Z_p $ with special fibre $\cX_{\F _p }$ and generic fibre $X_{\Q_p}$, and let $b,z$ be two $\Z_p$-points of $\cX$.

\par A proof of the following theorem, which we will refer to as \textit{nonabelian Berthelot-Ogus comparison} by analogy with \cite[Theorem 2.4]{BO83}, may be found in \cite[\S 2.4.1]{CLS99}.
\begin{theorem}[Chiarellotto--Le Stum]
The analytification functor defines an equivalence of categories 
\[
(-)^{\an} \ : \ \cC^{\dR}(X_{\Q_p}) \ \stackrel{\sim }{ \lra } \ \cC^{\rig}(\cX_{\F _p }).
\]
For any point $x\in X(\Q _p )$ with mod $p$ reduction $\overline{x}$, we have a canonical isomorphism of fibre functors on $\cC ^{\dR} (X_{\Q _p })$:
\[
\overline{x}^* \circ (-)^{\an } \simeq x^* .
\]
In particular, we obtain isomorphisms of universal objects for any $n\geq 1$:
\[
\cA_n ^{\dR}(b)^{\an } \simeq \cA_n^{\rig}(\overline{b}).
\]
\end{theorem}

\par Since the vector space $A_n ^{\dR} (b,x)$ carries a filtration, and that the vector space $A_n ^{\rig }(\overline{b},\overline{x})$ carries a Frobenius action, the non-abelian Berthelot-Ogus theorem gives $A_n ^{\dR} (b,x)$ the structure of a filtered $\phi $-module structure. Explicitly, given a Teichm{\"u}ller representatives $b_0 $ and $x_0 $ for $b$ and $x$, the $\phi $-module is given by the endomorphism of $x_0 ^* \cA ^{\dR} (b)$ constructed in Lemma \ref{lemma:too_far}, and the isomorphism between the filtered vector space $x^* \cA _n ^{\dR}(b)$ and the $\phi $-module $x_0 ^* \cA _n ^{\dR}(\overline{b})$ is given by the parallel transport isomorphisms $T_{x,x_0 }$ and $T_{b,b_0 }$.

\par The filtered $\phi$-module structure on $A_n^{\dR}(b,x)$ is related to the $p$-adic Galois representation $A_n^{\et}(b,x)$ by Olsson's comparison theorem \cite[Theorem 1.11]{Ols11}. 

\begin{theorem}[Olsson]
For $A_n^{\et}(b,x)$ and $A_n^{\dR} (b,x)$ as above, there is an isomorphism
\[
\D_{ \cris }(A_n ^{\et}(b,x)) \stackrel{ \sim }{ \lra }A_n^{\dR} (b,x)
\]
of filtered $\phi $-modules, which on graded pieces $A[n] := \Ker(A_n(b,x) \to A_{n-1}(b,x))$ induce commutative diagrams
\[
\begin{tikzpicture}
\matrix (m) [matrix of math nodes, row sep=2.5em,
column sep=3em, text height=1.5ex, text depth=0.25ex]
{\D_{ \cris }(V_{\et})^{\otimes n} & V_{\dR }^{\otimes n} \\
\D_{ \cris }(A^{\et}[n]) & A^{\dR}[n] \\};
\path[->]
(m-1-1) edge[auto] node[auto]{} (m-2-1)
edge[auto] node[auto] { } (m-1-2)
(m-2-1) edge[auto] node[auto] { } (m-2-2)
(m-1-2) edge[auto] node[auto] {} (m-2-2);
\end{tikzpicture}
\]

\end{theorem}

\bibliographystyle{alpha}
\bibliography{References}
\end{document}